\documentclass{article}

\newcounter{notecount}

\usepackage{amssymb}

\renewcommand{\ge}{\geqslant}
\renewcommand{\leq}{\leqslant}
\renewcommand{\geq}{\geqslant}

\usepackage[utf8]{inputenc}
\usepackage[big,online]{dgruyter}
\usepackage{microtype}
\usepackage{lmodern} 
\usepackage{graphicx}
\usepackage{amsmath}
\usepackage{amssymb}
\usepackage{multirow}
\usepackage{ushort}
\usepackage{authblk}
\usepackage{leftidx}
\usepackage{amsfonts}
\usepackage{booktabs} 
\usepackage{array} 
\usepackage{paralist} 
\usepackage{verbatim} 
\usepackage{slashbox}
\usepackage{here}
\usepackage{anyfontsize}
\usepackage{url}
\usepackage{amsthm}

\newcommand{\vertiii}[1]{{\left\vert\kern-0.25ex\left\vert\kern-0.25ex\left\vert #1 
		\right\vert\kern-0.25ex\right\vert\kern-0.25ex\right\vert}}

\newcommand{\Lnorm}[2]{\left(#1, #2\right)}
\newcommand{\DGnorm}[2]{a\left(#1, #2\right)}

\newcommand{\llnorm}[1]{\big\lVert #1\big\rVert_{L_2(\Omega)}}
\newcommand{\edgelnorm}[1]{\big\lVert #1\big\rVert_{L_2(e)}}
\newcommand{\elementlnorm}[1]{\big\lVert #1\big\rVert_{L_2(E)}}
\newcommand{\ilnorm}[1]{\left\lVert #1\right\rVert_{L_\infty(0,T;L_2(\Omega))}}
\newcommand{\Llnorm}[1]{\left\lVert #1\right\rVert_{L_2(0,T;L_2(\Omega))}}

\newcommand{\hsnorm}[1]{\vertiii{#1}_{H^1(\mathcal{E}_h)}}
\newcommand{\enorm}[1]{\big\lVert #1\big\rVert_{V}}

\newcommand{\gnorm}[1]{\big\lVert #1\big\rVert_{L_2(\Gamma_N)}}
\newcommand{\hgnorm}[1]{\left\lVert #1\right\rVert_{H^1(0,T;L_2(\Gamma_N))}}
\newcommand{\ignorm}[1]{\left\lVert #1\right\rVert_{L_\infty(0,T;L_2(\Gamma_N))}}
\newcommand{\lgnorm}[1]{\left\lVert #1\right\rVert_{L_2(0,T;L_2(\Gamma_N))}}
\newcommand{\sobolevl}[1]{\vertiii{#1}_{H^0(\mathcal{E}_h)}}
\newcommand{\Gnorm}[2]{\left(#1, #2\right)_{L_2(\Gamma_N)}}
\newcommand{\Nphi}{N_\varphi}

\newcommand{\dgelliptic}{\boldsymbol R}
\newcommand{\dgv}{H^s(\mathcal{E}_h)}
\newcommand{\dgvh}{\mathcal{D}_k(\mathcal{E}_h)}
\newcommand{\Dgv}{\boldsymbol{H}^s(\mathcal{E}_h)}
\newcommand{\Dgvh}{\boldsymbol{\mathcal{D}}_k(\mathcal{E}_h)}
\newcommand{\jump}[2]{J^{\alpha_0,\beta_0}_0\left(#1,#2\right)}
\newcommand{\njump}[1]{[#1\otimes\boldsymbol{n}_e]}

\newcommand{\edgesum}{\sum\limits_{e\subset\Gamma_h\cup\Gamma_D}}
\newcommand{\elementsum}{\sum\limits_{E\in\mathcal{E}_h}}
\newcommand{\sumE}{\sum\limits_{E\in\mathcal{E}_h}\int_E}
\newcommand{\sume}{\sum\limits_{e\subset\Gamma_h\cup\Gamma_D}\int_e}
\newcommand{\nmax}{\max\limits_{0\leq n\leq N}}
\newtheorem{thm}{Theorem}[section]
\newtheorem{cor}{Corollary}[section]
\newtheorem{lem}{Lemma}[section]

\newtheorem{rmk}{Remark}[section]
\newcommand{\sumq}{\sum\limits_{q=1}^{\Nphi}}
\numberwithin{equation}{section}

\allowdisplaybreaks
\begin{document}

  \articletype{Research Article}

  \author[1]{Yongseok Jang}
  \author*[2]{Simon Shaw}
  
  \runningauthor{Y.Jang and S.Shaw}
  \affil[1]{ONERA Centre de Chatillon,
29 avenue de la Division Leclerc, 92322, Chatillon, France\\ Email: \href{mailto:yongseok.jang@onera.fr}{yongseok.jang@onera.fr} }
  \affil[2]{Department of Mathematics, Brunel University London, Kingston Lane, Uxbridge, UB8 3PH, UK\\ Email: \href{mailto:simon.shaw@brunel.ac.uk}{simon.shaw@brunel.ac.uk}}
  \title{\textit{A Priori} Analysis of a Symmetric Interior Penalty Discontinuous
  Galerkin Finite Element Method for a Dynamic Linear Viscoelasticity Model}
  \runningtitle{SIPG Finite Element Method for Viscoelasticity}
  \abstract{
The stress-strain constitutive law for viscoelastic materials such as soft tissues,
metals at high temperature, and polymers, can be written as a Volterra integral equation of the
second kind with a \emph{fading memory} kernel. This integral relationship 
yields current stress for a given strain history and can be used
in the momentum balance law to derive a mathematical model for the resulting
deformation. We consider such a dynamic linear viscoelastic model problem
resulting from using a \textit{Dirichlet-Prony} series of decaying exponentials to
provide the fading memory in the Volterra kernel.
We introduce two types of \textit{internal variable} to replace the Volterra integral with a
system of auxiliary ordinary differential equations and then use a spatially discontinuous 
symmetric interior penalty Galerkin (SIPG) finite element method
and --- in time --- a Crank-Nicolson method to formulate the
fully discrete problems: one for each type of internal variable.
We present \textit{a priori} stability and error analyses without using Gr\"onwall's inequality, and with
the result that the constants in our estimates grow linearly with time rather than exponentially. 
In this sense the schemes are therefore suited to simulating long time viscoelastic response  and
this (to our knowledge) is the first time that such high quality estimates have been presented for 
SIPG finite element approximation of dynamic viscoelasticty problems. We also carry out a
number of numerical experiments using the FEniCS environment
(e.g. \url{https://fenicsproject.org})
and explain how the codes can be obtained and the results reproduced.}
  \received{...}
  \accepted{...}
  \journalname{Computational Methods in Applied Mathematics}
  \startpage{1}
  \DOI{...}

\maketitle

\begin{acknowledgement}
Y.~Jang acknowledges the support of a Brunel University London Doctoral scholarship.
\end{acknowledgement}

\begin{funding}
This research was funded, in whole or in part, by Brunel University London. A CC BY or equivalent
licence is applied to Author Accepted Manuscript (the AAM) arising from this submission,
in accordance with the University’s grant’s open access conditions.
\end{funding}

\section{Introduction} 
  The application of a \textit{nonsymmetric interior penalty Discontinuous Galerkin
  \textup{(NIPG)} Finite Element Method} (DGFEM) to a dynamic
 linear solid viscoelasticity problem with tensor-valued
 internal variable stress representation was
 presented by Rivi\`ere, Shaw and Whiteman in \cite{DGV}.
 They gave an \textit{a priori} energy error estimate by using the standard 
 Gr\"onwall inequality to deal with the time accumulation of error, and hence
 the constants in the stability and error bounds are too large to give confidence in the
 long time simulation of viscoelastic response. In this paper, we use the
 symmetric interior penalty Galerkin (SIPG) method and prove stability bounds and
 \textit{a priori} error estimates (not only in the energy norm but also in 
 the spatial $L_2$ norm) without the use of Gr\"onwall's inequality.
 We therefore obtain non-exponentially increasing bounds for temporal
 $L_\infty$-type norms.
  Furthermore, we introduce vector-valued internal variables in \textit{displacement form}
  and \textit{velocity form} (to be defined below). This has the advantage of reducing computer memory requirements in that we need only store vectors of dimension $d$, instead of
  symmetric second order tensors of dimension $d(d+1)/2$, as in \cite{DGV}. This can be significant for high fidelity 3D simulations.
 
We consider a linear homogeneous and isotropic viscoelastic solid material,
e.g. \cite{golden2013boundary}, occupying
a bounded polytopic domain, the interior of which is denoted by
$\Omega\subset\mathbb{R}^d$, and consider the deformation and stress-strain state of
this material over times $t\in[0,T]$, where $T>0$.
The deformation, $\boldsymbol{u}$, and stress, $\ushort{\boldsymbol{\sigma}}$, follow
the momentum equation,
\begin{equation}
\rho\ddot{\boldsymbol{u}}-\nabla\cdot\ushort{\boldsymbol{\sigma}}=\boldsymbol{f}
\qquad\text{on $\Omega\times(0,T]$},\label{eq:primal:visco}
\end{equation}
where overdots denote time differentiation so that $\ddot{\boldsymbol{u}}$ is acceleration,
$\rho$ is the mass density of the material (assumed constant),
$\nabla\cdot\ushort{\boldsymbol{\sigma}}$ is the divergence of stress and $\boldsymbol{f}$ is an external body force (e.g. see \cite{DGV,VE}). Similarly, $\dot{\boldsymbol{u}}$ denotes velocity.
In addition to this vector-valued governing equation, we assume a mix of essential and
natural boundary conditions so that
\begin{alignat}{2}
\boldsymbol{u}(t) &=0
&\qquad
\textrm{on }\Gamma_D\times[0,T],\label{bc:Dirichlet}
\\
\ushort{\boldsymbol{\sigma}}(t)\cdot\boldsymbol{n}&=\boldsymbol{g}_N(t)
&
\textrm{on }\Gamma_N\times[0,T],\label{bc:Neumann}
\end{alignat}
where $\Gamma_D$ is the \textit{Dirichlet} boundary (assumed to have positive surface measure),
$\Gamma_N$ is the \textit{Neumann} boundary given by
$\Gamma_N=\partial\Omega\backslash\Gamma_D$, $\boldsymbol{n}$ is an outward unit normal vector
defined a.e. on $\Gamma_N$, and $\boldsymbol{g}_N$ prescribes a surface traction on $\Gamma_N$. 
Furthermore, for initial conditions on the displacement and the velocity we take,
\begin{align}
\boldsymbol{u}(0)=\boldsymbol{u}_0
\qquad\text{and}\qquad
\dot{\boldsymbol{u}}(0)=\boldsymbol{w}_0
\label{ic}
\end{align}
for given functions $\boldsymbol{u}_0$ and $\boldsymbol{w}_0$.

To close the problem, and solve for displacement, we need a constitutive equation expressing
stress in terms of displacement. In the linear viscoelasticity model considered here this
involves a Volterra (or `fading memory') integral with the specific material
characterised by stiffness, $\ushort{\boldsymbol{D}}$, and a stress relaxation function, $\varphi$,
see e.g.~\cite{drozdov1998viscoelastic,findley2013creep,hunter1976mechanics,VE,golden2013boundary}.
The stress is then given by
\begin{equation}
\ushort{\boldsymbol{\sigma}}(t)=\ushort{\boldsymbol{D}}\varphi(t)\ushort{\boldsymbol{\varepsilon}}(0)+\int^t_0\ushort{\boldsymbol{D}}\varphi(t-s)\ushort{\dot{\boldsymbol{\varepsilon}}}(s)ds,\label{eq:consti:visco}
\end{equation}
where $\ushort{\boldsymbol{D}}$ is a fourth order positive definite tensor satisfying 
the symmeteries $D_{ijkl}=D_{jikl}=D_{ijlk}=D_{klij}$, and
$\ushort{\boldsymbol{\varepsilon}}$ is the strain defined by
\[
\varepsilon_{ij}(\boldsymbol{v})
=\frac{1}{2}
\left(\frac{\partial v_i}{\partial x_j}+\frac{\partial v_j}{\partial x_i}\right),
\qquad\text{for }i,j=1,\ldots,d.
\]
Note that in \eqref{eq:consti:visco} we use the shorthand
$\ushort{\boldsymbol{\varepsilon}}(t) = \ushort{\boldsymbol{\varepsilon}}(\boldsymbol{u}(t))$.
The form of $\varphi$ depends on which viscoelastic model is invoked. There are several (see
e.g.~\cite{drozdov1998viscoelastic,findley2013creep,golden2013boundary} and the references therein)
but here we focus on the \textit{Generalised Maxwell solid} where
\begin{align}
    \varphi(t)
    =\varphi_0+\sum_{q=1}^{N_\varphi}\varphi_qe^{-t/\tau_q}
\label{eq:stressrelax:exponential}
\end{align}
with $N_\varphi\in\mathbb{N}$, strictly positive delay times $\{\tau_q\}^{N_\varphi}_{q=1}$, and coefficients $\{\varphi_q\}^{N_\varphi}_{q=0}$ the latter of which are normalised so that
$\varphi(0)=1$. The positivity requirement excludes the case $\varphi_0=0$ (a fluid in the
sense used by Golden and Graham in \cite{golden2013boundary}): this is an important 
assumption in the arguments developed below.
 
This paper is arranged as follows. In Section~\ref{sec:preliminary} we give our notations and the preliminary background for DGFEM. In Section~\ref{sec:modelproblem} we introduce two forms
of internal variables, each of which are used to represent the Volterra (or `history')
integral, and formulate a variational problem for each form.
We then state and prove (without using Gr{\"o}nwall's inequality) stability bounds in
Section~\ref{sec:stabilityanalysis} and error bounds in Section~\ref{sec:erroranalysis},
carry out some illustrative numerical experiments in Section~\ref{sec:numerical experiments},
using FEniCS, see \cite{alnaes2015fenics} and \url{https://fenicsproject.org},
and then end with some concluding remarks in Section~\ref{sec:conclusion}.

\section{Preliminary}\label{sec:preliminary}
We use standard notation so that $L_p(\Omega),\ H^s(\Omega)$ and $W^s_p(\Omega)$ 
(with $s$ and $p$ non-negative) denote the usual Lebesgue, Hilbert and Sobolev spaces.
For any normed space $X$, $\lVert\cdot\rVert_X$ is the $X$ norm which, for inner product spaces, is
always the norm induced by the inner product. For example,
$\lVert\cdot\rVert_{L_2(\Omega)}$ is the $L_2(\Omega)$ norm, as induced by the
$L_2(\Omega)$ inner product denoted---for brevity---by $(\cdot,\cdot)$, but for
$S\subset\bar\Omega$, we use $(\cdot,\cdot)_{L_2(S)}$ for the $L_2(S)$ inner product.
For time dependent functions we expand this notation so that for $X$ a normed target space,
$f\in L_p(0,T;X)$ denotes the space of $L_p(0,T)\to X$ functions with norm
\[
\lVert f\rVert_{L_p(0,T;X)}=\left(\int^{T}_0\lVert f(t)\rVert_X^p\,dt\right)^{1/p}
\]
for $1\leq p<\infty$. When $p=\infty$ this becomes the \textit{essential supremum} norm:
\[
\lVert f\rVert_{L_\infty(0,T;X)}
=\mathop{\mathrm{ess~sup}}\limits_{0\leq t\leq T}\lVert f(t)\rVert_X .
\]
When convenient, we shall often replace the upper limit $T$ in these expressions by some other
value $t\in[0,T]$.

For inner products of vector-valued and tensor-valued functions we use the same notation
as for the scalar cases. For instance, we have
\[
\Lnorm{\boldsymbol{v}}{\boldsymbol{w}}
=\int_\Omega \boldsymbol{v}\cdot\boldsymbol{w}
\,d\Omega,
\qquad\Lnorm{\ushort{\boldsymbol{v}}}{\ushort{\boldsymbol{w}}}=\int_\Omega \ushort{\boldsymbol{v}}:\ushort{\boldsymbol{w}}\,d\Omega=\sum_{i,j=1}^{d}\int_\Omega v_{ij}w_{ij}\,d\Omega,
\]
for vector-valued functions $\boldsymbol{v}$ and $\boldsymbol{w}$, and second order tensors
$\ushort{\boldsymbol{v}}$ and $\ushort{\boldsymbol{w}}$.

\paragraph*{Meshes}
We refer to \cite{DG} for a detailed explanation of the framework of the DGFEM and here just
summarise the main points.
Assume that the closure of $\Omega$ is subdivided into closed elements $E$, where $E$ is a triangle in 2D or a
tetrahedron in 3D, and the intersection of any pair of elements is either a vertex,
an edge, a face, or empty. The diameter of $E$ is defined by
$h_E:=\sup\limits_{x,y\in E}\lVert x-y\rVert$, where $\lVert\cdot\rVert$ is the Euclidean norm,
and $|E|$ denotes the measure  (area/volume) of $E$. In a similar way, let $e$ be an edge of $E$
and use $|e|$ to denote its measure (length/area). 
Let $h$ be the maximum of the diameters $h_E$ over all the elements $E$, and define
the set $\mathcal{E}_h$ of all of those elements. Then
$|e|\leq h_E^{d-1}\leq h^{d-1}$ for all $e\subset \partial E$ for each $E\in\mathcal{E}_h$.
We further suppose that the subdivision is quasi-uniform, which means that there exists a positive
constant $C$ such that $h\leq Ch_E$ for all $E\in\mathcal{E}_h$.

Next, let $\Gamma_h$ be the set of interior edges (in 2D) or faces (in 3D) contained in
the subdivision $\mathcal{E}_h$. Then for each edge or face element $e$,
we can define a unit normal vector, $\boldsymbol{n}_e$.
If $e\subset \partial \Omega$, $\boldsymbol{n}_e$ is the outward unit normal vector. For an interior
edge $e$ such that $e\subset E_i\cap E_j$ with $i<j$, the normal vector $\boldsymbol{n}_e$ is oriented from $E_i$ to $E_j$.

\paragraph*{Test spaces}
We introduce the broken Sobolev space $H^s(\mathcal{E}_h)
=\big\{v\in L_2(\Omega)\ |\ \forall E\in \mathcal{E}_h,\ v|_{E}\in H^s(E)\big\}$ and endow it
with the broken Sobolev norm, $\vertiii{\cdot}_{H^s(\mathcal{E}_h)}$, defined by
\begin{align*}
\vertiii{v}_{H^s(\mathcal{E}_h)}=\left(\sum_{E\in\mathcal{E}_h}\lVert v\rVert_{H^s(E)}^2\right)^{1/2}.
\end{align*}
We note the following facts
$H^s(\Omega)\subset H^s(\mathcal{E}_h)$ and $H^{s+1}(\mathcal{E}_h)\subset H^s(\mathcal{E}_h)$.
These definitions and notations are extended in an obvious way to the
the vector field analogue $\Dgv$.

We define the space of polynomials of degree less than or equal to $k$ on $E$,
for $E\subset\mathbb{R}^d$, by
\begin{align*}
\mathcal{P}_k(E)
=\textrm{span}\left\{x_1^{i_1}\cdots x_d^{i_d}\ |\ \sum_{m=1}^di_m\leq k,\ \boldsymbol x \in E, i_m\in\mathbb{N}\cup\{0\} \text{ for each }m
\right\},
\end{align*}
and then define our DG finite element space as
\begin{align*}
\mathcal{D}_k(\mathcal{E}_h)=\big\{
v \in H^1(\mathcal{E}_h)\ \big|
\ v|_{E}\in\mathcal{P}_k(E)\ \textrm{for each $E\in \mathcal{E}_h$}\big\}.
\end{align*}
The analogous vector field is given by $\Dgvh:=[\dgvh]^d$.

\paragraph*{Average and Jump}
Suppose two elements $E^e_i$ and $E_j^e$ share the common edge $e$ with $i<j$ and that there is a vector
valued function $\boldsymbol{v}$ and a second order tensor
$\ushort{\boldsymbol{v}}$ on $E^e_i$ and $E_j^e$.
Then we define an average and a jump for $\boldsymbol{v}$ and $\ushort{\boldsymbol{v}}$ by
\begin{align*}
\{\boldsymbol{v}\}
=\frac{(\boldsymbol{v}|_{E_i^e})|_e+(\boldsymbol{v}|_{E^e_j})|_e}{2},
\qquad&
\{\ushort{\boldsymbol{v}}\}
=\frac{(\ushort{\boldsymbol{v}}|_{E_i^e})|_e+(\ushort{\boldsymbol{v}}|_{E^e_j})|_e}{2},
\\
[\boldsymbol{v}]
=(\boldsymbol{v}|_{E_i^e})|_e-(\boldsymbol{v}|_{E^e_j})|_e,
\qquad&
{[\boldsymbol{v}\otimes\boldsymbol{n}_e]}
=({\boldsymbol{v}}|_{E_i^e})|_e\otimes\boldsymbol{n}_e
-(\boldsymbol{v}|_{E^e_j})|_e\otimes\boldsymbol{n}_e
\end{align*}
where the normal vector $\boldsymbol{n}_e$ is oriented from $E_i^e$ to $E_j^e$
and $\otimes$ is the outer product defined, for vectors $\boldsymbol{a}$ and $\boldsymbol{b}$, by
$(\boldsymbol{a}\otimes\boldsymbol{b})_{mn}=a_{m}b_{n}$ for $m,n=1,\ldots,d$.
On the other hand, if $e\subset\partial\Omega$ and $e\subset\partial E$
\begin{align*}
\{\boldsymbol{v}\}=\boldsymbol{v}|_e,
\qquad
\{{\boldsymbol{\ushort{v}}}\}=\boldsymbol{\ushort{v}}|_e,\
\qquad
[\boldsymbol{v}]=\boldsymbol{v}|_e\cdot{\boldsymbol{n}_e},
\qquad\textrm{ and }\qquad
{[\boldsymbol{v}\otimes\boldsymbol{n}_e]}={\boldsymbol{v}}|_e\otimes\boldsymbol{n}_e.
\end{align*}
We can now introduce the jump penalty operator,
\[
\jump{\boldsymbol{v}}{\boldsymbol{w}}
=
\sum_{e\subset\Gamma_h\cup\Gamma_D}
\frac{\alpha_0}{|e|^{\beta_0}}\int_e[\boldsymbol v]\cdot[\boldsymbol w]\,de,
\]
where $\alpha_0$ and  $\beta_0$ are positive constants.

\paragraph*{Useful inequalities}
We now recall the following inequalities for use later in the \textit{a priori} analysis.
\begin{itemize}
    \item Inverse polynomial trace inequalities \cite{WARBURTON20032765}: For any $v\in\mathcal{P}_k(E)$, $\forall e\subset\partial E,$
    \begin{equation}\left\{\begin{array}{ll}
   \edgelnorm{v}&\leq C|e|^{1/2}|E|^{-1/2}\elementlnorm{v},\\
     \edgelnorm{v}&\leq Ch_E^{-1/2}\elementlnorm{v},\\
     \edgelnorm{{\nabla v\cdot\boldsymbol{n}_e}}&\leq C|e|^{1/2}|E|^{-1/2}\elementlnorm{\nabla v},\\
     \edgelnorm{\nabla v\cdot\boldsymbol{n}_e}&\leq Ch_E^{-1/2}\elementlnorm{\nabla v},
    \end{array}\right.\label{trace:inv}
\end{equation}
where $C$ is a positive constant and is independent of $h_E$ but depends on the polynomial degree $k$.

\item {Poincar\'e's Inequality \cite{brenner2003poincare,DG}}: If $\beta_0(d-1)\geq1$
and $|e|\leq1$ for every $e\subset\Gamma_h\cup\Gamma_D$, then,
\begin{align}\label{dgpoincare}
\llnorm{v}\leq C\left(\sobolevl{\nabla v}^2+\sum_{e\subset\Gamma_h\cup\Gamma_D}\frac{1}{|e|^{\beta_0}}\edgelnorm{[v]}^2\right)^{1/2},
\end{align}
for any $v\in H^1(\mathcal{E}_h)$.

\item {Inverse Inequality (or Markov Inequality) \cite{ozisik2010constants,DG}}: 
For any $E\in\mathcal{E}_h$, there is a positive constant $C$ such that
\begin{align}\label{dginverse}
    \forall v\in \mathcal{P}_k(E),\ \elementlnorm{\nabla^jv}\leq Ch_E^{-j}\elementlnorm{v},\ \forall j\in\{0,1,\ldots, k\},
\end{align}where 
\begin{align*}
    \nabla^jv=\left\{ \begin{array}{cc}
         \nabla\cdot\nabla^{j-1}v&\text{for even }j,  \\
         \nabla(\nabla^{j-1}v)&\text{for odd }j, 
    \end{array}\right. \qquad\text{and}\qquad\nabla^0v=v.
\end{align*}
\end{itemize}
Note that these three inequalities can also be applied in the obvious way to vector-valued
functions.

\paragraph*{DG bilinear form}
We define a DG bilinear form $ a:\Dgv\times \Dgv\mapsto\mathbb{R}$ for $s>3/2$ by
\begin{align}
\DGnorm{\boldsymbol v}{\boldsymbol w}
=&\sum_{E\in\mathcal{E}_h}\int_E
\ushort{\boldsymbol D}\ushort{\boldsymbol \varepsilon}(\boldsymbol  v)
:\ushort{\boldsymbol \varepsilon} (\boldsymbol w)\,dE
-\sum_{e\subset\Gamma_h\cup\Gamma_D}\int_e
\{\ushort{\boldsymbol D}\ushort{\boldsymbol \varepsilon}(\boldsymbol  v)\}
:{[\boldsymbol w\otimes\boldsymbol{n}_e]}\, de
\nonumber\\
&-\sum_{e\subset\Gamma_h\cup\Gamma_D}\int_e
\{\ushort{\boldsymbol D}\ushort{\boldsymbol \varepsilon}(\boldsymbol  w)\}
:{[\boldsymbol v\otimes\boldsymbol{n}_e]}\,de
+ J_0^{\alpha_0,\beta_0}(\boldsymbol v,\boldsymbol w),
\label{eq:DG_biform}
\end{align}
for any $\boldsymbol v,\boldsymbol w\in \Dgv$.
We also define our DG energy norm by
\[
\enorm{\boldsymbol{v}}
=\left(
\sumE\ushort{\boldsymbol{D}}\ushort{\boldsymbol{\varepsilon}}(\boldsymbol{v}):\ushort{\boldsymbol{\varepsilon}}(\boldsymbol{v})\,dE+\jump{\boldsymbol{v}}{\boldsymbol{v}}
\right)^{1/2},\qquad\textrm{for }\boldsymbol{v}\in\Dgv.
\]
Comparing these we can observe that
\begin{align}
    \DGnorm{\boldsymbol{v}}{\boldsymbol{v}}=\enorm{\boldsymbol{v}}^2
    -2\!\!\!\!\!\!\!\sume\{\ushort{\boldsymbol D}\ushort{\boldsymbol \varepsilon}(\boldsymbol  v)\}:\njump{\boldsymbol{v}}\,de.
    \label{eq:relation:dgnorm}
\end{align}
\begin{rmk}
In the DG bilinear form, the third term is called the ``\textit{interior penalty}'' term and the
last one is called the ``\textit{jump penalty}''. Depending on the sign of the interior penalty, the bilinear form
is either symmetric or nonsymmetric. In this article, we consider only the symmetric DG method and refer to \textnormal{\cite{jang2020spatially,DGV}} for an application of the nonsymmetric method for viscoelasticity.
The reason why we employ SIPG is that it requires only the standard penalisation, $\beta_0(d-1)\ge1$, for optimal
spatial error estimates, while NIPG needs the super penalisation $\beta_0(d-1)\geq3$.
Since the use of the super penalisation enforces the linear system to be more ill-conditioned, we may encounter some difficulty in solving the system with iterative solvers. For more details, we refer to \textnormal{\cite{jang2020spatially}}.
\end{rmk}
\begin{rmk}[Korn's inequality for piecewise $H^1$ vector fields \cite{brenner2004korn,DG}] 
If we have $\beta_0(d-1)\geq1$, then since $\ushort{\boldsymbol{D}}$ is symmetric positive definite and the jump penalty is defined not only on the interior edges but also on the positive measured Dirichlet boundary, Korn's inequality yields, for any $\boldsymbol{v}\in \boldsymbol{H}^1(\mathcal{E}_h)$,
\begin{align}
    \elementsum\elementlnorm{\nabla\boldsymbol{v}}^2\leq C\enorm{\boldsymbol{v}}^2,\label{eq:kordDG}
\end{align}for some positive $C$ independent of $\boldsymbol{v}$.
\end{rmk}

In CGFEM the squared energy norm, $\enorm{\boldsymbol{v}}^2$, is usually given by
the energy inner product, $\DGnorm{\boldsymbol{v}}{\boldsymbol{v}}$. Here, in DGFEM, we
see from \eqref{eq:relation:dgnorm} that this is nearly true but there is an extra term.
The following lemma allows us to deal with that term.

\begin{lem}
Suppose $\alpha_0>0$ and $\beta_0(d-1)\geq 1$. For any $\boldsymbol{v},\boldsymbol{w}\in\Dgvh$,
and for any pair $E_1, E_2\in \mathcal{E}_h$, we have
\begin{align}
    \bigg|\int_e\{\ushort{\boldsymbol D}\ushort{\boldsymbol \varepsilon}(\boldsymbol  v)\}:\njump{\boldsymbol{w}}\,de\bigg|
    \leq\frac{C}{\sqrt{\alpha_0}}\bigg(\lVert{\ushort{\boldsymbol D}\ushort{\boldsymbol \varepsilon}(\boldsymbol{v})}\rVert_{L_2(E_1)}^2+\lVert{\ushort{\boldsymbol D}\ushort{\boldsymbol \varepsilon}(\boldsymbol{v})}\rVert_{L_2(E_2)}^2+\frac{\alpha_0}{|e|^{\beta_0}}\edgelnorm{[\boldsymbol{w}]}^2\bigg),\label{eq:vector:dg:bdd1}
\end{align}
where $e$ is the shared edge of elements $E_1$ and $E_2$,
and $C$ is a positive constant independent of $\boldsymbol{v}$ and $\boldsymbol{w}$
but dependent on the inverse polynomial trace inequality's constants and the domain.
When $e\subset \Gamma_N$, \eqref{eq:vector:dg:bdd1} holds with $E_1=E_2$.
\end{lem}

\begin{proof}
Let $e\subset \Gamma_h$ and $e\subset E_1\cap E_2$ where $E_1,E_2\in\mathcal{E}_h$.
Recalling the definitions of the average and jump, and using the Cauchy-Schwarz inequality,
we have
\begin{align*}
    \bigg|\int_e&\{\ushort{\boldsymbol D}\ushort{\boldsymbol \varepsilon}(\boldsymbol  v)\}:\njump{\boldsymbol{w}}\,de\bigg|\leq\frac{1}{2}\left(\edgelnorm{\ushort{\boldsymbol D}\ushort{\boldsymbol \varepsilon}(\boldsymbol  v|_{E_1})}+\edgelnorm{\ushort{\boldsymbol D}\ushort{\boldsymbol \varepsilon}(\boldsymbol  v|_{E_2})}\right)\frac{|e|^{\beta_0/2}}{|e|^{\beta_0/2}}\edgelnorm{[\boldsymbol{w}]},
\end{align*}after noting that $\njump{\boldsymbol{w}}=[\boldsymbol{w}]\otimes\boldsymbol{n}_e$, since $\boldsymbol{n}_e|_{E_1}=-\boldsymbol{n}_e|_{E_2}$. The inverse polynomial trace inequality \eqref{trace:inv} implies that
\begin{align*}
\bigg|    \int_e&\{\ushort{\boldsymbol D}\ushort{\boldsymbol \varepsilon}(\boldsymbol  v)\}:\njump{\boldsymbol{w}}\,de\bigg|\leq C\left(|e|^{\beta_0/2}h^{-1/2}_{E_1}\lVert{\ushort{\boldsymbol D}\ushort{\boldsymbol \varepsilon}(\boldsymbol{v})}\rVert_{L_2(E_1)}+|e|^{\beta_0/2}h^{-1/2}_{E_2}\lVert{\ushort{\boldsymbol D}\ushort{\boldsymbol \varepsilon}(\boldsymbol{v})}\rVert_{L_2(E_2)}\right)\frac{\edgelnorm{[\boldsymbol{w}]}}{|e|^{\beta_0/2}},
\end{align*}
for a positive constant $C$. Since $|e|\leq h^{d-1}$, 
the discrete Cauchy-Schwarz inequality yields
\begin{align*}
    \bigg|\int_e&\{\ushort{\boldsymbol D}\ushort{\boldsymbol \varepsilon}(\boldsymbol  v)\}:\njump{\boldsymbol{w}}de\bigg|\\\leq&C\bigg(h^{\beta_0(d-1)-1}_{E_1}+h^{\beta_0(d-1)-1}_{E_2}\bigg)^{1/2}\bigg(\lVert{\ushort{\boldsymbol D}\ushort{\boldsymbol \varepsilon}(\boldsymbol{v})}\rVert_{L_2(E_1)}^2+\lVert{\ushort{\boldsymbol D}\ushort{\boldsymbol \varepsilon}(\boldsymbol{v})}\rVert_{L_2(E_2)}^2\bigg)^{1/2}\frac{\edgelnorm{[\boldsymbol{w}]}}{|e|^{\beta_0/2}},\\
    \leq&C\bigg(\lVert{\ushort{\boldsymbol D}\ushort{\boldsymbol \varepsilon}(\boldsymbol{v})}\rVert_{L_2(E_1)}^2+\lVert{\ushort{\boldsymbol D}\ushort{\boldsymbol \varepsilon}(\boldsymbol{v})}\rVert_{L_2(E_2)}^2\bigg)^{1/2}\frac{\edgelnorm{[\boldsymbol{w}]}}{|e|^{\beta_0/2}},
\end{align*}
because $h_{E_1}\leq\textrm{diam}(\Omega),\ h_{E_2}\leq\textrm{diam}(\Omega)$ and $\beta_0(d-1)\geq1$. Using Young's inequality we then obtain
\begin{align*}
\bigg|    \int_e \{\ushort{\boldsymbol D}\ushort{\boldsymbol \varepsilon}(\boldsymbol  v)\}:\njump{\boldsymbol{w}}de\bigg|
\leq
C\left(\frac{\epsilon}{2}
\left(
\lVert{\ushort{\boldsymbol D}\ushort{\boldsymbol \varepsilon}(\boldsymbol{v})}\rVert_{L_2(E_1)}^2+\lVert{\ushort{\boldsymbol D}\ushort{\boldsymbol \varepsilon}(\boldsymbol{v})}\rVert_{L_2(E_2)}^2
\right)+\frac{1}{2\epsilon}\frac{\edgelnorm{[\boldsymbol{w}]}^2}{|e|^{\beta_0}}\right),
\end{align*}
for any positive $\epsilon$. Taking $\epsilon=1/\sqrt{\alpha_0}$ then completes the proof.
\end{proof}

\begin{cor}\label{cor:estimate_interior}
By \eqref{eq:vector:dg:bdd1}, we can also derive
\begin{align}
    \sume\bigg|\{\ushort{\boldsymbol D}\ushort{\boldsymbol \varepsilon}(\boldsymbol  v)\}:\njump{\boldsymbol{w}}de\bigg|
    \leq\frac{C}{\sqrt{\alpha_0}}\left(\enorm{\boldsymbol{v}}^2+\jump{\boldsymbol{w}}{\boldsymbol{w}}\right),\label{eq:vector:dg:bdd2}
\end{align}
where $C$ is a positive constant independent of $\boldsymbol{v},\boldsymbol{w}\in\Dgvh$
but dependent on the inverse polynomial trace inequality's constant in \eqref{trace:inv}.
\end{cor}

\begin{thm}[Coercivity and continuity]\label{thm:coer_conti}
Suppose $\alpha_0>0$ is sufficiently large and $\beta_0(d-1)\geq1$. Then there exist positive constants $\kappa$ and $K$ such that
\begin{align*}
    \kappa\enorm{\boldsymbol{v}}^2\leq\DGnorm{\boldsymbol{v}}{\boldsymbol{v}},\qquad\left|\DGnorm{\boldsymbol{v}}{\boldsymbol{w}}\right|\leq K\enorm{\boldsymbol{v}}\enorm{\boldsymbol{w}},\qquad\forall \boldsymbol{v},\boldsymbol{w}\in\Dgvh,
\end{align*}
where $\kappa$ and $K$ are independent of $\boldsymbol{v}$ and $\boldsymbol{w}$.
\end{thm}

\begin{proof}
The proof follows the same arguments in \cite[Theorems 4.5 and 4.6]{jang2020spatially}. For example, the use of \eqref{eq:vector:dg:bdd1} and \eqref{eq:vector:dg:bdd2} leads us to show the coercivity and the continuity. For details, please see \cite{jang2020spatially}.
\end{proof}

\begin{rmk}
Due to the use of inverse polynomial trace inequality, the DG bilinear form will not be coercive and
continuous on the broken Sobolev space. In other words, \textnormal{Theorem \ref{thm:coer_conti}} holds
only on the finite element space. For the choice of the penalty parameter $\alpha_0$, we refer to \textnormal{\cite{hansbo2002discontinuous,wihler2006locking}}. For instance, we will take 
$\alpha_0\in[10,100]$ in the numerical experiments \textnormal{Section~\ref{sec:numerical experiments}}.
\end{rmk}

\paragraph*{DG elliptic projection}
The DG elliptic projector, $\dgelliptic$, is defined for $\boldsymbol u\in\Dgv$ and $s>3/2$ by,
\[
\dgelliptic:\Dgv\mapsto\Dgvh\ \textrm{ such that }\ \DGnorm{\boldsymbol u}{\boldsymbol v}=\DGnorm{\dgelliptic \boldsymbol u}{\boldsymbol v},\ \forall{{\boldsymbol v}}\in\Dgvh.
\]
Referring to \cite{DG,riviere2003discontinuous,houston2006hp,wihler2006locking}, for example,
we recall the following elliptic-error estimates, 
\begin{align}\label{vector:appDG}
     \enorm{\boldsymbol u-\dgelliptic \boldsymbol u}\leq{C}h^{\min(k+1,s)-1}\vertiii{\boldsymbol u}_{H^s(\mathcal{E}_h)},\\
    \label{vector:appl2}
    \llnorm{\boldsymbol u-\dgelliptic \boldsymbol u}\leq{C}h^{\min(k+1,s)}\vertiii{\boldsymbol u}_{H^s(\mathcal{E}_h)},
\end{align}
for $\boldsymbol u\in \Dgv$ with $s>3/2$ and for sufficiently large penalty parameters
$\alpha_0\textrm{ and }\beta_0\geq(d-1)^{-1}$. Here, the positive constant $C$ is independent of
$\boldsymbol{u}$ but dependent on the domain, its boundary, and the polynomial degree $k$.

We now move on to describe the model problem.

\section{Model Problem} \label{sec:modelproblem}
We recall our primal equations \eqref{eq:primal:visco}-\eqref{eq:stressrelax:exponential}
and hereafter assume the data terms are bounded and smooth so that
$\boldsymbol{f}\in C(0,T;\boldsymbol L_2(\Omega))$,
$\boldsymbol{g}_N\in C^1(0,T;\boldsymbol L_2(\Gamma_N))$,
$\boldsymbol{u}_0\in \boldsymbol H^s(\Omega)$, and
$\boldsymbol{w}_0\in \boldsymbol L_2(\Omega)$.
We first of all set up the internal variable representations of viscoelasticity and
then give the variational formulations.

\subsection{Internal Variables}
Recalling the constitutive equation \eqref{eq:consti:visco} and the stress relaxation function \eqref{eq:stressrelax:exponential} we define internal variables as
\begin{align}
\boldsymbol{\psi}_q(t)
=\int^{t}_{0} \frac{\varphi_q}{\tau_q}e^{-(t-t')/\tau_q}\boldsymbol u(t')\,dt'
\qquad\textrm{ and }\qquad
\boldsymbol{\zeta}_q(t)
=\int^{t}_{0} \varphi_qe^{-(t-t')/\tau_q}\dot{\boldsymbol u}(t')\,dt'
\label{eq:internal}
\end{align}
for $q=1,\ldots,N_\varphi$, and note that these are zero at $t=0$. Using \eqref{eq:stressrelax:exponential} in \eqref{eq:consti:visco},
and the Leibniz integral rule leads us to two alternative forms of the constitutive equation,
\begin{align}
\ushort{\boldsymbol\sigma}(\boldsymbol u(t))
&=
\ushort{{\boldsymbol {D}}}\ushort{{\boldsymbol\varepsilon}}\bigg(\varphi_0\boldsymbol u(t)+\sum_{q=1}^{N_\varphi}\boldsymbol{\zeta}_q(t)\bigg)+\sum_{q=1}^{N_\varphi} \varphi_qe^{-t/\tau_q}\ushort{{\boldsymbol {D}}}\ushort{{\boldsymbol\varepsilon}}(\boldsymbol u_0),
\label{Constitutive:velo}
\\
\text{or }\qquad
\ushort{\boldsymbol\sigma}(\boldsymbol u(t))
&= \ushort{\boldsymbol D}\ushort{\boldsymbol\varepsilon}\bigg(\boldsymbol u(t)-\sum_{q=1}^{N_\varphi}\boldsymbol{\psi}_q(t)\bigg),\label{Constitutive:disp}
\end{align}
where we noted that 
$\boldsymbol{\zeta}_q(t)+\varphi_q e^{-t/\tau_q}\boldsymbol{u}(0)
=\varphi_q\boldsymbol{u}(t) - \boldsymbol{\psi}_q(t)$.
We call $\boldsymbol{\psi}_q$ (resp. $\boldsymbol{\zeta}_q$) the \textit{displacement form}
(resp. \textit{velocity form}) of the internal variable. It is easy to check that \eqref{Constitutive:disp} and \eqref{Constitutive:velo} are equivalent by integration by parts
and remembering that $\sumq\varphi_q=1-\varphi_0$.
Using these constitutive relationships we can derive dynamic linear viscoelasticity model
problems in two forms as follows.

\paragraph*{Displacement form}
We consider \eqref{eq:primal:visco} in the form:
find $\boldsymbol{u}$, $\boldsymbol{\psi}_1$, $\boldsymbol{\psi}_2$, \ldots, such that 
\begin{align}
&\rho \ddot{\boldsymbol u}-\nabla\cdot\ushort{\boldsymbol D}\ushort{\boldsymbol\varepsilon}\left(\boldsymbol{u}-\sum_{q=1}^{N_\varphi}\boldsymbol{\psi}_q\right)=\boldsymbol f,\label{primal:internalDisp}
\\
&\tau_q\dot{\boldsymbol{\psi}}_q+\boldsymbol{\psi}_q=\varphi_q\boldsymbol{u},
\text{ with }\boldsymbol{\psi}(0)=\boldsymbol{0},
\qquad\textrm{ for }q=1,\ldots,\Nphi.\label{primal:ode:disp}
\end{align}

\paragraph*{Velocity form}
We consider \eqref{eq:primal:visco} in the form:
\begin{align}
&\rho \ddot{\boldsymbol u}-\nabla\cdot\ushort{{\boldsymbol {D}}}\ushort{{\boldsymbol\varepsilon}}\left(\varphi_0\boldsymbol u+\sum_{q=1}^{N_\varphi}\boldsymbol{\zeta}_q\right)
=
\boldsymbol f+\nabla\cdot\left(\sum_{q=1}^{N_\varphi} \varphi_qe^{-t/\tau_q}\ushort{{\boldsymbol {D}}}\ushort{{\boldsymbol\varepsilon}}(\boldsymbol u_0)\right),
\label{primal:internalVelo}
\\
&\tau_q\dot{\boldsymbol{\zeta}}_q+\boldsymbol{\zeta}_q=\tau_q\varphi_q\dot{\boldsymbol{u}}
\text{ with }\boldsymbol{\zeta}(0)=\boldsymbol{0},
\qquad\textrm{for }q=1,\ldots,\Nphi.\label{primal:ode:velo}
\end{align}
In these we note that the auxiliary equations \eqref{primal:ode:disp} and \eqref{primal:ode:velo} are derived by 
time-differentiation of each of \eqref{eq:internal}.

\begin{rmk}
The (well known) idea here is to replace the integral form in the constitutive hereditary
laws by ODE's for a set of internal variables. With these we use a time stepper rather than numerical
integration to compute the discrete solution. Furthermore, while \textnormal{\cite{DGV}} employed internal
variables as second order tensors, we have defined vector-valued internal variables.
This reduces the computer memory requirement by a factor of $d$.
\end{rmk}

\subsection{Variational Formulation}
We begin by multiplying each of \eqref{primal:internalDisp} and \eqref{primal:internalVelo}
by test functions $\boldsymbol{v}\in H^s(\mathcal{E}_h)$, for $s>3/2$, integrating
by parts over each $E$, summing over every $E$, and then collecting up terms to form the
edge average and jump terms. We then assume sufficient spatial continuity of the continuous
solution to \eqref{primal:internalDisp} and \eqref{primal:ode:velo} so that we can add the
terms involving the edge jumps of the exact solution: 
terms three and four on the right of \eqref{eq:DG_biform}. This produces the DG bilinear form
on the left hand side and gives us a weak form for each choice of the form of the internal 
variables. The procedure is standard in the DGFEM literature, and for more details in this
specific setting we refer to \textnormal{\cite{jang2020spatially}}
(where it may be useful to note that $\ushort{\boldsymbol{\sigma}}\cdot\boldsymbol{n}\cdot\boldsymbol{v}=\ushort{\boldsymbol{\sigma}}:(\boldsymbol{v}\otimes\boldsymbol{n})$).

\paragraph*{Weak forms}
Let us define linear forms by
\[
F_d(t;\boldsymbol v)
=\Lnorm{\boldsymbol f(t)}{\boldsymbol v}+
\Gnorm{\boldsymbol g_N(t)}{\boldsymbol v},
\]
and
\[F_v(t;\boldsymbol v)=\Lnorm{\boldsymbol f(t)}{\boldsymbol v}+\Gnorm{\boldsymbol g_N(t)}{\boldsymbol v}-\sumq\varphi_qe^{-t/\tau_q}\DGnorm{\boldsymbol u_0}{\boldsymbol v}.\]
Recalling \eqref{eq:DG_biform}, we now use these linear forms to formulate variational problems for
the displacement form and the velocity form. First for the displacement form of the problem\ldots 

\textbf{Displacement Form problem, (D)}
find $\boldsymbol{u}$ and $\{\boldsymbol{\psi}_q\}_{q=1}^{\Nphi}$
such that for all $\boldsymbol{v}\in\Dgv$, with $s>3/2$, we have,
\begin{gather}
    \Lnorm{\rho\ddot{\boldsymbol{u}}(t)}{\boldsymbol{v}}+\DGnorm{\boldsymbol{u}(t)}{\boldsymbol v}-\sumq\DGnorm{\boldsymbol{\psi}_q(t)}{\boldsymbol v}+\jump{\dot{\boldsymbol{ u}}(t)}{\boldsymbol{v}}=F_d(t;\boldsymbol v),\label{eq:vector:dgfem:s1:primal1}\\
    \DGnorm{\tau_q\dot{\boldsymbol{\psi}}_q(t)+{\boldsymbol{\psi}}_q(t)}{{\boldsymbol{v}}}=\DGnorm{\varphi_q\boldsymbol{u}(t)}{\boldsymbol{v}},\ \text{for each $q$,}\label{eq:vector:dgfem:s1:primal2}
\end{gather}
where $\boldsymbol{u}(0)=\boldsymbol{u}_0$, $\dot{\boldsymbol{u}}(0)=\boldsymbol{w}_0$ and $\boldsymbol{\psi}_q(0)=\boldsymbol{0}$.
And, secondly, for the velocity form\ldots

\textbf{Velocity Form problem, (V)}
find $\boldsymbol{u}$ and $\{\boldsymbol{\zeta}_q\}_{q=1}^{\Nphi}$ such that
for all $\boldsymbol{v}\in\Dgv$, with $s>3/2$, we have,
\begin{gather}
    \hspace{-0.4cm}\Lnorm{\rho\ddot{\boldsymbol{u}}(t)}{\boldsymbol{v}}+\varphi_0\DGnorm{\boldsymbol{u}(t)}{\boldsymbol v}+\sumq\DGnorm{\boldsymbol{\zeta}_q(t)}{\boldsymbol v}+\jump{\dot{\boldsymbol{ u}}(t)}{\boldsymbol{v}}=F_v(t;\boldsymbol v),\label{eq:vector:dgfem:s2:primal1}\\
    \DGnorm{\tau_q\dot{\boldsymbol{\zeta}}_q(t)+{\boldsymbol{\zeta}}_q(t)}{{\boldsymbol{v}}}=\DGnorm{\tau_q\varphi_q\dot{\boldsymbol{u}}(t)}{\boldsymbol{v}},\ \text{for each $q$,}\label{eq:vector:dgfem:s2:primal2}
\end{gather} where $\boldsymbol{u}(0)=\boldsymbol{u}_0$, $\dot{\boldsymbol{u}}(0)=\boldsymbol{w}_0$ and $\boldsymbol{\zeta}_q(0)=\boldsymbol{0}$.

\paragraph*{Fully discrete formulations}
For the time discretisation we employ a \textit{Crank-Nicolson} type finite difference
scheme \cite{Crank1996}. Let $t_n=n\Delta t$ with $\Delta t = T/N$ for some
$N\in\mathbb{N}$, and denote the fully discrete approximations to
$\boldsymbol u$ and $\dot{\boldsymbol{u}}$
by $\boldsymbol u(t_n)=\boldsymbol u^n\approx \boldsymbol{U}_h^n\in \Dgvh$ and
$\dot{\boldsymbol {u}}(t_n)=\dot{\boldsymbol {u}}^n\approx \boldsymbol W_h^n\in \Dgvh$.
Similarly, for the internal variables we introduce
$\boldsymbol{\psi}_q(t_n)\approx\boldsymbol{\Psi}_{hq}^n\in\Dgvh$
and $\boldsymbol{\zeta}_q(t_n)\approx\boldsymbol{\mathcal{S}}_{hq}^n\in\Dgvh$, for each $q$.
Furthermore, in the schemes that follow we will use the following approximations, 
\[
\frac{\ddot {\boldsymbol {u}}(t_{n+1})+\ddot {\boldsymbol {u}}(t_{n})}{2}\approx\frac{\boldsymbol W^{n+1}_h-\boldsymbol W^{n}_h}{\Delta t}
\textrm{ and }
\frac{ \boldsymbol u(t_{n+1})+ \boldsymbol u(t_{n})}{2}\approx\frac{\boldsymbol U^{n+1}_h+\boldsymbol U^n_h}{2},
\]
and we will impose the relation,
\begin{equation}\label{r1}
  \frac{\boldsymbol W^{n+1}_h+\boldsymbol W^{n}_h}{2}=\frac{\boldsymbol U^{n+1}_h-\boldsymbol U^n_h}{\Delta t}.  
\end{equation}
We can now give the fully discrete schemes.

\textbf{Displacement Form problem, $\textbf{(D)}^h$}
find $\boldsymbol{W}_h^{n}$, $\boldsymbol{U}_h^{n}$, $\boldsymbol{\Psi}_{h1}^n,\ldots,\boldsymbol{\Psi}_{h\Nphi}^n\in\Dgvh$ for $n=0,\ldots,N$ such that for all $\boldsymbol{v}\in\Dgvh$
we have for $n=0,1,\ldots,N-1$ firstly that,
\begin{gather}
     \Lnorm{\rho\frac{\boldsymbol{W}_h^{n+1}-\boldsymbol{W}_h^{n}}{\Delta t}}{\boldsymbol{v}}+\DGnorm{\frac{\boldsymbol{U}_h^{n+1}+\boldsymbol{U}_h^{n}}{2}}{\boldsymbol v}-\sumq\DGnorm{\frac{\boldsymbol{\Psi}_{hq}^{n+1}+\boldsymbol{\Psi}_{hq}^{n}}{2}}{\boldsymbol v}\nonumber\\+\jump{\frac{\boldsymbol{W}_h^{n+1}+\boldsymbol{W}_h^{n}}{2}}{\boldsymbol{v}}=\frac{1}{2}\Big(F^{n+1}_d(\boldsymbol v)+F^n_d(\boldsymbol v)\Big),
    \label{eq:vector:dgfem:s1e1}
\end{gather}
and secondly that for $q=1,2,\ldots,\Nphi$,
\begin{gather}
     \DGnorm{\tau_q\frac{\boldsymbol{\Psi}_{hq}^{n+1}-\boldsymbol{\Psi}_{hq}^{n}}{\Delta t}+\frac{\boldsymbol{\Psi}_{hq}^{n+1}+\boldsymbol{\Psi}_{hq}^{n}}{2}}{{\boldsymbol{v}}}=\DGnorm{\varphi_q\frac{\boldsymbol{U}_h^{n+1}+\boldsymbol{U}_h^{n}}{2}}{\boldsymbol{v}},\label{eq:vector:dgfem:s1e2}
\\
\intertext{with initial data}     
    \label{eq:vector:dgfem:s1e3}\DGnorm{\boldsymbol{U}_h^{0}}{\boldsymbol{v}}=\DGnorm{\boldsymbol{u}_0}{\boldsymbol{v}},
\\
\label{eq:vector:dgfem:s1e4}\Lnorm{\boldsymbol{W}_h^{0}}{\boldsymbol{v}}=\Lnorm{\boldsymbol{w}_0}{\boldsymbol{v}},
\\
\boldsymbol{\Psi}_{hq}^0=\boldsymbol{0}, \ \forall{q}=1, \ldots,\Nphi,
\label{eq:vector:dgfem:s1e5}
\end{gather}
and where $F^n_d(\boldsymbol v)=F_d(t_n;\boldsymbol v)$.

\textbf{Velocity Form problem, $\textbf{(V)}^h$}
find $\boldsymbol{W}_h^{n}$, $\boldsymbol{U}_h^{n}$, $\boldsymbol{\mathcal{S}}_{h1}^n,\ldots,\boldsymbol{\mathcal{S}}_{h\Nphi}^n\in\Dgvh$ for $n=0,\ldots,N$ such that for all
$\boldsymbol{v}\in\Dgvh$ we have for $n=0,1,\ldots,N-1$ firstly that,
\begin{gather}
    \Lnorm{\rho\frac{\boldsymbol{W}_h^{n+1}-\boldsymbol{W}_h^{n}}{\Delta t}}{\boldsymbol{v}}+\varphi_0\DGnorm{\frac{\boldsymbol{U}_h^{n+1}+\boldsymbol{U}_h^{n}}{2}}{\boldsymbol v}+\sumq\DGnorm{\frac{\boldsymbol{\mathcal{S}}_{hq}^{n+1}+\boldsymbol{\mathcal{S}}_{hq}^{n}}{2}}{\boldsymbol v}\nonumber\\
    \label{eq:vector:dgfem:s2e1}+\jump{\frac{\boldsymbol{W}_h^{n+1}+\boldsymbol{W}_h^{n}}{2}}{\boldsymbol{v}}=\frac{1}{2}\Big( F_v^{n+1}(\boldsymbol v)+ F_v^n(\boldsymbol v)\Big),
    \end{gather}
and secondly that for $q=1,2,\ldots,\Nphi$,
    \begin{gather}
    \DGnorm{\tau_q\frac{\boldsymbol{\mathcal{S}}_{hq}^{n+1}-\boldsymbol{\mathcal{S}}_{hq}^{n}}{\Delta t}+\frac{\boldsymbol{\mathcal{S}}_{hq}^{n+1}+\boldsymbol{\mathcal{S}}_{hq}^{n}}{2}}{{\boldsymbol{v}}}=\DGnorm{\tau_q\varphi_q\frac{\boldsymbol{W}_h^{n+1}+\boldsymbol{W}_h^{n}}{2}}{\boldsymbol{v}},\label{eq:vector:dgfem:s2e2}
\\
\intertext{with initial data \eqref{eq:vector:dgfem:s1e3}, \eqref{eq:vector:dgfem:s1e4},}     
\boldsymbol{\mathcal{S}}_{hq}^0=\boldsymbol{0}, \ \forall{q}=1, \ldots,\Nphi,
\label{eq:vector:dgfem:s2e5}
\end{gather}
and where $F^n_v(\boldsymbol v)=F_v(t_n;\boldsymbol v)$.

\section{Stability Analysis}\label{sec:stabilityanalysis}
In this section we derive stability bounds for the solutions to the fully discrete displacement
and velocity forms of the problem. As the problems are linear, these stability bounds can also be
used to conclude the existence and uniqueness of the numerical solutions and,
hence, the well-posedness of the discrete problems. 

\begin{thm}[Stability bound for $\textbf{(D)}^h$]\label{thm:vector:dgfem:s1:stability}
If $\beta_0(d-1)\geq1$ and $\alpha_0$ is large enough, and if
$\boldsymbol{W}_h^{n}$, $\boldsymbol{U}_h^{n}$ and $\{\boldsymbol{\Psi}_{hq}^n\}_{q=1}^{\Nphi}$
in $\Dgvh$ satisfy the fully discrete formulation $\textnormal{\textbf{(D)}}^h$, then 
there exists a positive constant $C$ such that
\begin{gather*}
    \max_{0\leq n\leq N}\llnorm{\boldsymbol W^{n}_h}^2+\max_{0\leq n\leq N}\enorm{\boldsymbol U^{n}_h}^2+\sum_{q=1}^{N_\varphi}\max_{0\leq n\leq N}\enorm{\boldsymbol \Psi^{n}_{hq}}^2+\sum_{n=0}^{N-1}\sum_{q=1}^{N_\varphi}\frac{1}{\Delta t}\enorm{\boldsymbol \Psi^{n+1}_{hq}-\boldsymbol \Psi^{n}_{hq}}^2\\+\Delta t\sum_{n=0}^{N-1}\jump{\boldsymbol W_h^{n+1}+\boldsymbol W_h^n}{\boldsymbol W_h^{n+1}+\boldsymbol W_h^n}\\
    \leq CT^2\bigg(\llnorm{\boldsymbol w_0}^2+\enorm{\boldsymbol u_0}^2+\ilnorm{\boldsymbol f}^2+h^{-1}\hgnorm{\boldsymbol g_N}^2\bigg).
\end{gather*}
Here, $C$ is independent of the discrete solutions, $\Delta t$, and $h$, but dependent on the domain
$\Omega$, its boundary, and the material properties.
\end{thm}

\begin{proof}
Choose $m\in\mathbb{N}$ so that $m<N$.
Take $\boldsymbol v=\Delta t(\boldsymbol W^{n+1}_h+\boldsymbol W^n_h)$ in
\eqref{eq:vector:dgfem:s1e1} and
$\boldsymbol v=2(\boldsymbol \Psi_{hq}^{n+1}-\boldsymbol \Psi_{hq}^{n})$ in \eqref{eq:vector:dgfem:s1e2}, for each $q$, and for $n=0,\ldots,m-1$, and then add the results and
sum over $n$ to get,
\begin{align}
    {\rho}&\llnorm{\boldsymbol W^{m}_h}^2+\DGnorm{\boldsymbol U^{m}_h}{\boldsymbol U^{m}_h}+\sum_{q=1}^{N_\varphi}\frac{1}{\varphi_q}\DGnorm{\boldsymbol \Psi^{m}_{hq}}{\boldsymbol \Psi^{m}_{hq}}\nonumber+\sum_{n=0}^{m-1}\sum_{q=1}^{N_\varphi}\frac{2\tau_q}{\Delta t\varphi_q}\DGnorm{\boldsymbol \Psi^{n+1}_{hq}-\boldsymbol \Psi^{n}_{hq}}{\boldsymbol \Psi^{n+1}_{hq}-\boldsymbol \Psi^{n}_{hq}}\nonumber\\&+\frac{\Delta t}{2}\sum_{n=0}^{m-1}\jump{\boldsymbol W_h^{n+1}+\boldsymbol W_h^n}{\boldsymbol W_h^{n+1}+\boldsymbol W_h^n}\nonumber\\
    =&{\rho}\llnorm{\boldsymbol W^{0}_h}^2+\DGnorm{\boldsymbol U^{0}_h}{\boldsymbol U^{0}_h}+2\sum_{q=1}^{N_\varphi}\DGnorm{\boldsymbol \Psi^{m}_{hq}}{\boldsymbol U^{m}_h}\nonumber\\&+\frac{\Delta t}{2}\sum_{n=0}^{m-1}\left( F^{n+1}_d(\boldsymbol W^{n+1}_h+\boldsymbol W^{n}_h)+F^n_d(\boldsymbol W^{n+1}_h+\boldsymbol W^{n}_h)\right)
    \label{eq:vector:dgfem:sta:s1:eq1}
\end{align}
since, by \eqref{r1},
$\Delta t(\boldsymbol W^{n+1}_h+\boldsymbol W^n_h)
=2(\boldsymbol U^{n+1}_h-\boldsymbol{U}^n_h)$ and, by \eqref{eq:vector:dgfem:s1e5},
$\boldsymbol \Psi_{hq}^0=\boldsymbol0$, and also where we noted that
\[
\DGnorm{\boldsymbol \Psi^{n+1}_{hq}+\boldsymbol \Psi^{n}_{hq}}{\boldsymbol U^{n+1}_h-\boldsymbol U^{n}_h}
=
2\DGnorm{\boldsymbol \Psi^{n+1}_{hq}}{\boldsymbol U^{n+1}_h}
-2\DGnorm{\boldsymbol \Psi^{n}_{hq}}{\boldsymbol U^{n}_h}
-\DGnorm{\boldsymbol \Psi^{n+1}_{hq}-\boldsymbol \Psi^{n}_{hq}}{\boldsymbol U^{n+1}_h+\boldsymbol U^{n}_h}
\]
which, when used with \eqref{eq:vector:dgfem:s1e2}, gives us,
\begin{gather*}
    \DGnorm{\boldsymbol \Psi^{n+1}_{hq}+\boldsymbol \Psi^{n}_{hq}}{\boldsymbol U^{n+1}_h-\boldsymbol U^{n}_h}
    =2\DGnorm{\boldsymbol \Psi^{n+1}_{hq}}{\boldsymbol U^{n+1}_h}-2\DGnorm{\boldsymbol \Psi^{n}_{hq}}{\boldsymbol U^{n}_h}
    \\-
    \frac{2\tau_q}{\Delta t\,\varphi_q}\DGnorm{\boldsymbol \Psi^{n+1}_{hq}-\boldsymbol \Psi^{n}_{hq}}{\boldsymbol \Psi^{n+1}_{hq}-\boldsymbol \Psi^{n}_{hq}}
    -\frac{1}{\varphi_q}\left(\DGnorm{\boldsymbol \Psi^{n+1}_{hq}}{\boldsymbol \Psi^{n+1}_{hq}}-\DGnorm{\boldsymbol \Psi^{n}_{hq}}{\boldsymbol \Psi^{n}_{hq}}\right)
\end{gather*}
for each $n$ and $q$. 
Using \eqref{eq:relation:dgnorm}, \eqref{eq:vector:dgfem:sta:s1:eq1} and
Theorem~\ref{thm:coer_conti} in \eqref{eq:vector:dgfem:sta:s1:eq1} we obtain,
\begin{gather}
{\rho}\llnorm{\boldsymbol W^{m}_h}^2+\enorm{\boldsymbol U^{m}_h}^2+\sum_{q=1}^{N_\varphi}\frac{1}{\varphi_q}\enorm{\boldsymbol \Psi^{m}_{hq}}^2
\nonumber+\sum_{n=0}^{m-1}\sum_{q=1}^{N_\varphi}\frac{2\kappa\tau_q}{\Delta t\varphi_q}\enorm{\boldsymbol \Psi^{n+1}_{hq}-\boldsymbol \Psi^{n}_{hq}}^2
\nonumber\\
+\frac{\Delta t}{2}\sum_{n=0}^{m-1}\jump{\boldsymbol W_h^{n+1}+\boldsymbol W_h^n}{\boldsymbol W_h^{n+1}+\boldsymbol W_h^n}
\nonumber\\
\leq\bigg|{\rho}\llnorm{\boldsymbol W^{0}_h}^2+\DGnorm{\boldsymbol U^{0}_h}{\boldsymbol U^{0}_h}+2\sum_{q=1}^{N_\varphi}\DGnorm{\boldsymbol \Psi^{m}_{hq}}{\boldsymbol U^{m}_h}
\nonumber\\
+\frac{\Delta t}{2}\sum_{n=0}^{m-1}\left( F^{n+1}_d(\boldsymbol W^{n+1}_h+\boldsymbol W^{n}_h)+F^n_d(\boldsymbol W^{n+1}_h+\boldsymbol W^{n}_h)\right)
\nonumber\\
+2\sume\{\ushort{\boldsymbol D}\ushort{\boldsymbol \varepsilon}(\boldsymbol  {U}^m_h)\}
\colon\njump{\boldsymbol{U}^m_h}de
+\sumq\frac{2}{\varphi_q}\sume\{\ushort{\boldsymbol D}\ushort{\boldsymbol \varepsilon}(\boldsymbol  {\Psi}_{hq}^m)\}:\njump{\boldsymbol{\Psi}_{hq}^m}de\bigg|.
\label{eq:vector:dgfem:sta:s1:eq2}
\end{gather}
Now we show an upper bound for the right hand side of \eqref{eq:vector:dgfem:sta:s1:eq2}.
\begin{itemize}[$\bullet$]
    \item $\llnorm{\boldsymbol W^{0}_h}^2$: By the Cauchy-Schwarz inequality, \eqref{eq:vector:dgfem:s1e4} yields
    \[\llnorm{\boldsymbol W^{0}_h}^2=\Lnorm{\boldsymbol W^{0}_h}{\boldsymbol W^{0}_h}=\Lnorm{\boldsymbol W^{0}_h}{\boldsymbol w_0}\leq\llnorm{\boldsymbol W^{0}_h}\llnorm{\boldsymbol w_0},\]
    hence $\llnorm{\boldsymbol W^{0}_h}^2\leq\llnorm{\boldsymbol w_0}^2$.
    \item $\big|\DGnorm{\boldsymbol U^{0}_h}{\boldsymbol U^{0}_h}\big|$: Combining the coercivity and the continuity, we can derive
    \[\kappa\enorm{\boldsymbol U^{0}_h}^2\leq\DGnorm{\boldsymbol U^{0}_h}{\boldsymbol U^{0}_h}=\DGnorm{\boldsymbol U^{0}_h}{\boldsymbol u_0}\leq K\enorm{\boldsymbol U^{0}_h}\enorm{\boldsymbol u_0},\]
    by \eqref{eq:vector:dgfem:s1e3}. This implies $\enorm{\boldsymbol U^{0}_h}\leq K/\kappa\enorm{\boldsymbol u_0}$
    and so we have
    \[\big|\DGnorm{\boldsymbol U^{0}_h}{\boldsymbol U^{0}_h}\big|\leq K\enorm{\boldsymbol U^{0}_h}^2\leq\bar K\enorm{\boldsymbol u_0}^2\qquad\textrm{for $\bar K=\frac{K^3}{\kappa^2}$}.\]
    \item $\left|2\sum\limits_{q=1}^{N_\varphi}\DGnorm{\boldsymbol \Psi^{m}_{hq}}{\boldsymbol U^{m}_h}\right|$:
    The Cauchy-Schwarz inequality, Young's inequality and \eqref{eq:vector:dg:bdd2} yield
    \begin{align*}
        \bigg|2\sum_{q=1}^{N_\varphi}\DGnorm{\boldsymbol\Psi^{m}_{hq}}{\boldsymbol U^{m}_h}\bigg|\   \leq&\sum_{q=1}^{N_\varphi}(\epsilon_q+\frac{C}{\sqrt{\alpha_0}})\enorm{\boldsymbol U^{m}_h}^2+\sum_{q=1}^{N_\varphi}(\frac{1}{\epsilon_q}+\frac{C}{\sqrt{\alpha_0}})\enorm{\boldsymbol\Psi^{m}_{hq}}^2,
    \end{align*}for any positive $\epsilon_q$, $\forall q$.
    \item $\left|\frac{\Delta t}{2}\sum\limits_{n=0}^{m-1}\left( F^{n+1}_d(\boldsymbol W^{n+1}_h+\boldsymbol W^{n}_h)+F^n_d(\boldsymbol W^{n+1}_h+\boldsymbol W^{n}_h)\right)\right|$: Note that summation by parts and the fundamental theorem of calculus give
\begin{align*}
\sum_{n=0}^{m-1}\big(\boldsymbol g_N(t_{n+1})+\boldsymbol g_N(t_{n}),\boldsymbol U^{n+1}_h-\boldsymbol U^{n}_h\big)_{L_2(e)}
= &
2\big(\boldsymbol g_N(t_{m}),\boldsymbol U^{m}_h\big)_{L_2(e)}-2\big(\boldsymbol g_N(t_{0}),\boldsymbol U^{0}_h\big)_{L_2(e)}
\\
&-\sum_{n=0}^{m-1}\int^{t_{n+1}}_{t_{n}}\big(\dot {\boldsymbol g}_N(t'),\boldsymbol U^{n+1}_h+\boldsymbol U^n_h\big)_{L_2(e)}dt',
\qquad  \forall e\subset\Gamma_N.
\end{align*}
By \eqref{r1} and the Cauchy-Schwarz inequality, we therefore have
\begin{align*}
    \bigg|\frac{\Delta t}{2}&\sum_{n=0}^{m-1}\left( F^{n+1}_d(\boldsymbol W^{n+1}_h+\boldsymbol W^{n}_h)+F^n_d(\boldsymbol W^{n+1}_h+\boldsymbol W^{n}_h)\right)\bigg|\\\leq&\frac{\Delta t}{2}\sum_{n=0}^{m-1}\big(\llnorm{\boldsymbol {f}^{n+1}}+\llnorm{\boldsymbol {f}^n}\big)\llnorm{\boldsymbol W^{n+1}_h+\boldsymbol W^n_h}\\&+2\sum_{e\subset\Gamma_N}\edgelnorm{\boldsymbol  g_N^{m}}{\edgelnorm{\boldsymbol U^{m}_h}}+2\sum_{e\subset\Gamma_N}\edgelnorm{\boldsymbol  g_N^{0}}\edgelnorm{\boldsymbol U^{0}_h}\\&+\sum_{n=0}^{m-1}\int^{t_{n+1}}_{t_{n}}\sum_{e\subset\Gamma_N}\edgelnorm{\dot {\boldsymbol g}_N(t')}\edgelnorm{\boldsymbol U^{n+1}_h+\boldsymbol U^{n}_h}dt'.
\end{align*}
Since the inverse polynomial trace and Poincar\'e's inequalities imply that
\[\edgesum\edgelnorm{\boldsymbol v}^2\leq Ch^{-1}\elementsum\elementlnorm{\boldsymbol{v}}^2\leq Ch^{-1}\enorm{\boldsymbol{v}}^2, \forall \boldsymbol{v}\in\Dgvh,\]
the triangle and Young's inequalities lead us to,
\begin{align*}
    \bigg|\frac{\Delta t}{2}&\sum_{n=0}^{m-1}\left( F^{n+1}_d(\boldsymbol W^{n+1}_h+\boldsymbol W^{n}_h)+F^n_d(\boldsymbol W^{n+1}_h+\boldsymbol W^{n}_h)\right)\bigg|
    \\
    \leq&\frac{\Delta t}{4\epsilon_a}\sum_{n=0}^{m-1}\big(\llnorm{\boldsymbol {f}^{n+1}}^2+\llnorm{\boldsymbol {f}^n}^2\big)
    +{2\Delta t\epsilon_a}\sum_{n=0}^{m-1}\big(\llnorm{\boldsymbol W^{n+1}_h}^2+\llnorm{\boldsymbol W^{n}_h}^2\big)\\&+\frac{1}{\epsilon_b}\gnorm{\boldsymbol  g_N^{m}}^2+\frac{C\epsilon_b}{h}{\enorm{\boldsymbol U^{m}_h}^2}+\frac{1}{h}\gnorm{\boldsymbol  g_N^{0}}^2+C{\enorm{\boldsymbol U^{0}_h}^2}\\&+\frac{1}{\epsilon_b}\int^{t_{m}}_{0}\gnorm{\dot {\boldsymbol g}_N(t')}^2dt'+\frac{C\Delta t\epsilon_b}{2h}\sum_{n=0}^{m-1}(\enorm{\boldsymbol U^{n+1}_h}^2+\enorm{\boldsymbol U^{n}_h}^2),
\end{align*}for any positive $\epsilon_a$ and $\epsilon_b$.
Maximizing over time we then arrive at,
\begin{align*}
    \bigg|\frac{\Delta t}{2}&\sum_{n=0}^{m-1}\left( F^{n+1}_d(\boldsymbol W^{n+1}_h+\boldsymbol W^{n}_h)+F^n_d(\boldsymbol W^{n+1}_h+\boldsymbol W^{n}_h)\right)\bigg|
    \\
    \leq
    &
    \frac{T}{2\epsilon_a}\ilnorm{\boldsymbol {f}}^2+{4T\epsilon_a}\max_{0\leq n\leq N}\llnorm{\boldsymbol W^{n}_h}^2
    +\left(\frac{1}{\epsilon_b}+\frac{1}{h}\right)\ignorm{\boldsymbol  g_N}^2
    \\&
    +C{\enorm{\boldsymbol U^{0}_h}^2}
    +\frac{1}{\epsilon_b}\lgnorm{\dot {\boldsymbol g}_N}^2+\frac{C(T+1)\epsilon_b}{h}\nmax\enorm{\boldsymbol U^{n}_h}^2.
\end{align*}
\item $\left|2\sume\{\ushort{\boldsymbol D}\ushort{\boldsymbol \varepsilon}(\boldsymbol  {U}^m_h)\}:\njump{\boldsymbol{U}^m_h}de\right|$: \eqref{eq:vector:dg:bdd2} implies that
\[\left|2\sume\{\ushort{\boldsymbol D}\ushort{\boldsymbol \varepsilon}(\boldsymbol  {U}^m_h)\}:\njump{\boldsymbol{U}^m_h}de\right|\leq \frac{C}{\sqrt{\alpha_0}}\enorm{\boldsymbol{U}^m_h}^2.\]
\item $\left|\sumq\frac{2}{\varphi_q}\sume\{\ushort{\boldsymbol D}\ushort{\boldsymbol \varepsilon}(\boldsymbol  {\Psi}_{hq}^m)\}:\njump{\boldsymbol{\Psi}_{hq}^m}de\right|$: In the same manner, \eqref{eq:vector:dg:bdd2} yields
\[\left|\sumq\frac{2}{\varphi_q}\sume\{\ushort{\boldsymbol D}\ushort{\boldsymbol \varepsilon}(\boldsymbol  {\Psi}_{hq}^m)\}:\njump{\boldsymbol{\Psi}_{hq}^m}de\right|\leq\sumq\frac{C}{\varphi_q\sqrt{\alpha_0}}\enorm{\boldsymbol{\Psi}_{hq}^m}^2.\]
\end{itemize}
Collecting these results together then gives,
\begin{gather}
    {\rho}\llnorm{\boldsymbol W^{m}_h}^2+\left(1-\sumq\epsilon_q-\frac{C\Nphi}{\sqrt{\alpha_0}}\right)\enorm{\boldsymbol U^{m}_h}^2\nonumber\\+\sum_{q=1}^{N_\varphi}\left(\frac{1}{\varphi_q}-\frac{1}{\epsilon_q}-\frac{C}{\sqrt{\alpha_0}}\right)\enorm{\boldsymbol \Psi^{m}_{hq}}^2+\sum_{n=0}^{m-1}\sum_{q=1}^{N_\varphi}\frac{2\kappa\tau_q}{\Delta t\varphi_q}\enorm{\boldsymbol \Psi^{n+1}_{hq}-\boldsymbol \Psi^{n}_{hq}}^2\nonumber\\+\frac{\Delta t}{2}\sum_{n=0}^{m-1}\jump{\boldsymbol W_h^{n+1}+\boldsymbol W_h^n}{\boldsymbol W_h^{n+1}+\boldsymbol W_h^n}\nonumber\\
    \leq{\rho}\llnorm{\boldsymbol w_0}^2+(\bar K+ CK/\kappa)\enorm{\boldsymbol u_0}^2+\frac{T}{2\epsilon_a}\ilnorm{\boldsymbol {f}}^2\nonumber\\+\left(\frac{1}{\epsilon_b}+\frac{1}{h}\right)\ignorm{\boldsymbol  g_N}^2+\frac{1}{\epsilon_b}\lgnorm{\dot {\boldsymbol g}_N}^2\nonumber\\+{4T\epsilon_a}\max_{0\leq n\leq N}\llnorm{\boldsymbol W^{n}_h}^2+\frac{C(T+1)\epsilon_b}{h}\nmax\enorm{\boldsymbol U^{n}_h}^2.\label{eq:vector:dgfem:sta:s1:eq3}
\end{gather}When we set $\epsilon_q=\varphi_q+\varphi_0/(2\Nphi)>0$ for each $q$, \eqref{eq:vector:dgfem:sta:s1:eq3} gives\begin{gather}
    {\rho}\llnorm{\boldsymbol W^{m}_h}^2+\left(\frac{\varphi_0}{2}-\frac{C\Nphi}{\sqrt{\alpha_0}}\right)\enorm{\boldsymbol U^{m}_h}^2\nonumber\\+\sum_{q=1}^{N_\varphi}\left(\frac{\varphi_0}{2\Nphi\varphi_q^2+\varphi_0\varphi_q}-\frac{C}{\sqrt{\alpha_0}}\right)\enorm{\boldsymbol \Psi^{m}_{hq}}^2+\sum_{n=0}^{m-1}\sum_{q=1}^{N_\varphi}\frac{2\kappa\tau_q}{\Delta t\varphi_q}\enorm{\boldsymbol \Psi^{n+1}_{hq}-\boldsymbol \Psi^{n}_{hq}}^2\nonumber\\+\frac{\Delta t}{2}\sum_{n=0}^{m-1}\jump{\boldsymbol W_h^{n+1}+\boldsymbol W_h^n}{\boldsymbol W_h^{n+1}+\boldsymbol W_h^n}\nonumber\\
    \leq{\rho}\llnorm{\boldsymbol w_0}^2+(\bar K+ CK/\kappa)\enorm{\boldsymbol u_0}^2+\frac{T}{2\epsilon_a}\ilnorm{\boldsymbol {f}}^2\nonumber\\+\left(\frac{1}{\epsilon_b}+\frac{1}{h}\right)\ignorm{\boldsymbol  g_N}^2+\frac{1}{\epsilon_b}\lgnorm{\dot {\boldsymbol g}_N}^2\nonumber\\+{4T\epsilon_a}\max_{0\leq n\leq N}\llnorm{\boldsymbol W^{n}_h}^2+\frac{C(T+1)\epsilon_b}{h}\nmax\enorm{\boldsymbol U^{n}_h}^2.\label{eq:vector:dgfem:sta:s1:eq4}
\end{gather}
After noting that the right hand side of \eqref{eq:vector:dgfem:sta:s1:eq4} is independent of $m$,
and that $m$ is arbitrary, we can obtain
\begin{gather}
    \frac{\rho}{2}\nmax\llnorm{\boldsymbol W^{n}_h}^2
    +\left(\frac{\varphi_0}{4}-\frac{C\Nphi}{\sqrt{\alpha_0}}\right)\nmax\enorm{\boldsymbol U^{n}_h}^2
    +\sum_{q=1}^{N_\varphi}\left(\frac{\varphi_0}{2\Nphi\varphi_q^2+\varphi_0\varphi_q}-\frac{C}{\sqrt{\alpha_0}}\right)\nmax\enorm{\boldsymbol \Psi^{n}_{hq}}^2
    \nonumber\\
    +\sum_{n=0}^{N-1}\sum_{q=1}^{N_\varphi}\frac{2\kappa\tau_q}{\Delta t\varphi_q}\enorm{\boldsymbol \Psi^{n+1}_{hq}-\boldsymbol \Psi^{n}_{hq}}^2
    +\frac{\Delta t}{2}\sum_{n=0}^{N-1}\jump{\boldsymbol W_h^{n+1}+\boldsymbol W_h^n}{\boldsymbol W_h^{n+1}+\boldsymbol W_h^n}
    \nonumber\\
    \leq{\rho}\llnorm{\boldsymbol w_0}^2+(\bar K+ CK/\kappa)\enorm{\boldsymbol u_0}^2+\frac{4T^2}{\rho}\ilnorm{\boldsymbol {f}}^2
    \nonumber\\
    +\frac{1}{h}\left(\frac{4C(T+1)}{\varphi_0}+1\right)\ignorm{\boldsymbol  g_N}^2
    +\frac{4C(T+1)}{\varphi_0h}\lgnorm{\dot {\boldsymbol g}_N}^2,
    \label{eq:vector:dgfem:sta:s1:eq5}
\end{gather}where $\epsilon_a=\rho/(8T)$ and $\epsilon_b=\varphi_0h/(4C(T+1))$.
Taking $\alpha_0$ sufficiently large then completes the proof.
\end{proof}
\begin{thm}[Stability bound for $\textbf{(V)}^h$]\label{thm:vector:dgfem:s2:stability}
If $\beta_0(d-1)\geq1$ and $\alpha_0$ is large enough, and if
$\boldsymbol{W}_h^{n}$, $\boldsymbol{U}_h^{n}$
and $\{\boldsymbol{\mathcal{S}}_{hq}^n\}_{q=1}^{\Nphi}$ satisfy the fully discrete formulation 
$\textnormal{\textbf{(V)}}^h$, then there exists a positive constant $C$ such that
\begin{gather*}
    \max_{0\leq n\leq N}\llnorm{\boldsymbol W^{n}_h}^2+\max_{0\leq n\leq N}\enorm{\boldsymbol{\mathcal{S}} ^{n}_h}^2+\sum_{q=1}^{N_\varphi}\max_{0\leq n\leq N}\enorm{\boldsymbol \Psi^{n}_{hq}}^2+\sum_{n=0}^{N-1}\sum_{q=1}^{N_\varphi}{\Delta t}\enorm{\boldsymbol {\mathcal{S}} ^{n+1}_{hq}-\boldsymbol{\mathcal{S}} ^{n}_{hq}}^2\\+\Delta t\sum_{n=0}^{N-1}\jump{\boldsymbol W_h^{n+1}+\boldsymbol W_h^n}{\boldsymbol W_h^{n+1}+\boldsymbol W_h^n}\\
    \leq CT^2\bigg(\llnorm{\boldsymbol w_0}^2+\enorm{\boldsymbol u_0}^2+\ilnorm{\boldsymbol f}^2+h^{-1}\hgnorm{\boldsymbol g_N}^2\bigg).
\end{gather*}
Here, $C$ is independent of discrete solutions, $\Delta t$ and $h$ but dependent on the domain
$\Omega$, its boundary, and the material properties.
\end{thm}
\begin{proof}
The proof follows the same arguments in Theorem \ref{thm:vector:dgfem:s1:stability}.
We take $\boldsymbol{v}=\Delta t(\boldsymbol{W}^{n+1}_h+\boldsymbol{W}^{n}_h)$ in \eqref{eq:vector:dgfem:s2e1} and $\boldsymbol{v}=\Delta t(\boldsymbol{\mathcal{S}}^{n+1}_h+\boldsymbol{\mathcal{S}}^{n+1}_h)$ in \eqref{eq:vector:dgfem:s2e2}, for each $q$, and
add the results. We then use \eqref{dgpoincare}, the inverse polynomial trace inequality, \eqref{eq:relation:dgnorm}, \eqref{eq:vector:dg:bdd2}, the Cauchy-Schwarz and Young's inequalities,
continuity and coercivity of the DG bilinear form, and maximise over discrete time just as before.
There is this additional term in $F_v(\cdot)$,
\[\sum\limits_{n=0}^{m-1}\sumq\varphi_q(e^{-t_{n+1}/\tau_q}+e^{-t_n/\tau_q})\DGnorm{\boldsymbol{u}_0}{\boldsymbol{U}^{n+1}_h-\boldsymbol{U}_h^n}
\]
which can be dealt with by using the fundamental theorem of calculus,
summation by parts and continuity. Specifically,
\begin{align*}
    \bigg|&\sum\limits_{n=0}^{m-1}\sumq\varphi_q(e^{-t_{n+1}/\tau_q}+e^{-t_n/\tau_q})\DGnorm{\boldsymbol{u}_0}{\boldsymbol{U}^{n+1}_h-\boldsymbol{U}_h^n}\bigg|
    \\=&
    \bigg|2\sumq\varphi_q e^{-t_m/\tau_q}\DGnorm{\boldsymbol{u}_0}{\boldsymbol{U}^{m}_h}-2\sumq\varphi_q e^{-t_0/\tau_q}\DGnorm{\boldsymbol{u}_0}{\boldsymbol{U}^{0}_h}
    \\&
    \qquad-\sum\limits_{n=0}^{m-1}\sumq\varphi_q(e^{-t_{n+1}/\tau_q}-e^{-t_n/\tau_q})\DGnorm{\boldsymbol{u}_0}{\boldsymbol{U}^{n+1}_h+\boldsymbol{U}_h^n}\bigg|\\
    \leq&2\big|\DGnorm{\boldsymbol{u}_0}{\boldsymbol U^{m}_h}\big|+2\big|a\big({\boldsymbol u_0},{\boldsymbol U^{0}_h}\big)\big|+2K\sum_{q=1}^{N_\varphi}\varphi_q\sum_{n=0}^{m-1}\int^{t_{n+1}}_{t_{n}}\frac{e^{-t'/\tau_q}}{\tau_q}\enorm{\boldsymbol u_0}\enorm{\boldsymbol U^{n}_h}dt',
\end{align*}since $|e^{-t/\tau_q}|\leq1,\  \forall t\geq0$ for each $q$ and $\sumq\varphi_q\leq1$.
Young's inequality, continuity and maximising over time now yield an $m$-independent upper bound.
The proof is now completed in a similar way (appropriate choice of Young's inequality constants,
and large enough $\alpha_0$) to the proof of Theorem \ref{thm:vector:dgfem:s1:stability}.
\end{proof}

\begin{rmk}
Note that in both of the stability estimates, \textnormal{Theorems~\ref{thm:vector:dgfem:s1:stability}
and~\ref{thm:vector:dgfem:s2:stability}}, the factor $h^{-1}$ appears in the traction term.
This unwelcome dependence is not unusual, see for example \textnormal{\cite{DG,DGV}}, and
it arises due to the technical arguments. However, it is not observed in practical computations. 
\end{rmk}


\section{Error Analysis}\label{sec:erroranalysis}
In order to carry out an \textit{a priori} error analysis we take the usual approach
of splitting the error into `spatial' and `temporal' components by using the elliptic
projection. To this end, define 
\begin{gather*}
    \boldsymbol \theta(t):=\boldsymbol u(t)-\boldsymbol R\boldsymbol u(t),\quad \boldsymbol \vartheta_q(t):=\boldsymbol \psi_q(t)-\boldsymbol R\boldsymbol \psi_q(t),\quad \boldsymbol \nu_q(t):=\boldsymbol \zeta_q(t)-\boldsymbol{R}\boldsymbol \zeta_q(t), 
\\    \boldsymbol \chi^n:=\boldsymbol U_h^n-\boldsymbol R\boldsymbol u^n,\qquad 
    \boldsymbol \varpi^n:=\boldsymbol W_h^n- \boldsymbol R\dot {\boldsymbol u}^n,\quad\boldsymbol \varsigma^n_q:=\boldsymbol \Psi_{hq}^n-\boldsymbol R \boldsymbol \psi_q^n,
     \\       \boldsymbol \Upsilon^n_q:=\mathcal{\boldsymbol S}_{hq}^n-\boldsymbol R\boldsymbol \zeta_q^n,\quad\text{ and }\quad\boldsymbol {\mathcal{ E}}_1(t):=\frac{\ddot {\boldsymbol u}{(t+\Delta t)}+\ddot {\boldsymbol u}(t)}{2}-\frac{\dot {\boldsymbol u}{(t+\Delta t)}-\dot{\boldsymbol  u}(t)}{\Delta t},
\end{gather*}for $t\in[0,T]$, $q=1,\ldots,\Nphi$, and $n=0,\ldots,N$,
and note that \eqref{r1} implies that
\begin{align}
    \frac{\boldsymbol \chi^{n+1}-\boldsymbol \chi^{n}}{\Delta t}=&\frac{\boldsymbol \varpi^{n+1}+\boldsymbol \varpi^{n}}{2}-\boldsymbol {\mathcal{E}}_2^n-\boldsymbol {\mathcal{E}}_3^n\label{eq:vector:dgfem:dgr},
\end{align}
for $n=0,\ldots,N-1$ where
\begin{align*}
    \boldsymbol {\mathcal{E}}_2(t)&:=\frac{\dot{\boldsymbol  \theta}(t+\Delta t)+\dot {\boldsymbol \theta}(t)}{2}-\frac{\boldsymbol \theta(t+\Delta t)-\boldsymbol \theta(t)}{\Delta t},
    \\
    \boldsymbol {\mathcal{E}}_3(t)&:=\frac{\boldsymbol u(t+\Delta t)-\boldsymbol u(t)}{\Delta t}-\frac{\dot {\boldsymbol u}(t+\Delta t)+\dot {\boldsymbol u}(t)}{2}.
\end{align*}
Also, for a three-times time-differentiable function, $\boldsymbol v(t)$, with
$\boldsymbol v^{(3)}$ denoting the third time derivative, we have
	\begin{align*}
	\frac{\dot {\boldsymbol v}(t_{n+1})+\dot{\boldsymbol  v}(t_{n})}{2}-\frac{\boldsymbol v(t_{n+1})-\boldsymbol v(t_{n})}{\Delta t}=\frac{1}{2\Delta t}\int^{t_{n+1}}_{t_n}\boldsymbol v^{(3)}(t)(t_{n+1}-t)(t-t_n)dt.
\end{align*}
Hence, if  $\boldsymbol v^{(3)}\in{L_2(t_n,t_{n+1};\boldsymbol{X})}$, the
Cauchy-Schwarz inequality gives
	\begin{equation}
	\left\lVert{\frac{\dot{\boldsymbol  v}(t_{n+1})+\dot {\boldsymbol v}(t_{n})}{2}-\frac{\boldsymbol v(t_{n+1})-\boldsymbol v(t_{n})}{\Delta t}}\right\rVert_{X}^2\leq\frac{\Delta t^3}{4}\lVert \boldsymbol v^{(3)}\rVert_{L_2(t_n,t_{n+1};X)}^2.
	\label{eq:error:CN}
	\end{equation}

As is well known, the usual path to an error bound is to use the triangle inequality
to split the error using the spatial and temporal components introduced above. The spatial
errors are bounded using standard results for elliptic problems, and so it is the temporal error
components that demand the bulk of the effort. For our problems, this effort is contained in the 
next two lemmas: the first for $\textbf{(D)}^h$ and the second for $\textbf{(V)}^h$.

\begin{lem}
\label{lemma:vector:dgfem:error:s1}
Suppose
$\boldsymbol u\in H^2(0,T;\boldsymbol{C}^2({\Omega}))\cap W^1_\infty(0,T;\Dgv)\cap H^4(0,T;\Dgv)$ and $\beta_0(d-1)\geq 1$ for $s>3/2$. If, for $n=0,\ldots,N$, the solution to
\textnormal{$\textbf{(D)}^h$} is $\boldsymbol{U}^n_h$, $\boldsymbol{W}^n_h$,
$\boldsymbol{\Psi}^n_{h1},\ldots,\boldsymbol{\Psi}^n_{h\Nphi}$ then,
for large enough $\alpha_0$, there exists a positive constant $C$ such that\begin{gather*}
    \max_{1\leq n\leq N}\llnorm{\boldsymbol \varpi^{n}}+\max_{1\leq n\leq N}\enorm{\boldsymbol \chi^{n}}\leq CT\lVert\boldsymbol{u}\rVert_{H^4(0,T;\dgv)}(h^{r}+\Delta t^2),
\end{gather*}
where $r=\min{(k+1,s)}$ and $C$ are independent of $h$, $\Delta t$, $T$ and the numerical solution.
\end{lem}

\begin{proof}
The proof will follow similar arguments in the stability analysis but we additionally
use Galerkin orthogonality, the elliptic error estimates \eqref{vector:appDG} and \eqref{vector:appl2}, and the time
discretisation error \eqref{eq:error:CN}. Averaging over $t_{n+1}$ and $t_n$ and subtracting
\eqref{eq:vector:dgfem:s1:primal1} from \eqref{eq:vector:dgfem:s1e1} give
\begin{align}
    \frac{\rho}{\Delta t}&\Lnorm{\boldsymbol \varpi^{n+1}-\boldsymbol \varpi^n}{\boldsymbol v}+\frac{1}{2}\DGnorm{\boldsymbol \chi^{n+1}+\boldsymbol \chi^n}{\boldsymbol v}-\frac{1}{2}\sum_{q=1}^{N_\varphi}\DGnorm{\boldsymbol \varsigma_q^{n+1}+\boldsymbol \varsigma_q^n}{\boldsymbol v}+\frac{1}{2}\jump{\boldsymbol \varpi^{n+1}+\boldsymbol \varpi^n}{\boldsymbol v}\nonumber\\
    =&
    \frac{\rho}{\Delta t}\Lnorm{\dot{\boldsymbol \theta}^{n+1}-\dot{\boldsymbol \theta}^n}{\boldsymbol v}+\rho\Lnorm{\boldsymbol {\mathcal{ E}}^n_1}{\boldsymbol v},\label{eq:vector:dgfem:error:s1:e1}
\end{align}$\forall\boldsymbol v\in \Dgvh$ for $0\leq n\leq N-1$. In this manner, a subtraction of \eqref{eq:vector:dgfem:s1e2} from \eqref{eq:vector:dgfem:s1:primal2} gives
\begin{align}
    \frac{\tau_q}{\Delta t}&\DGnorm{\boldsymbol\varsigma_q^{n+1}-\boldsymbol \varsigma_q^n}{\boldsymbol v}+\frac{1}{2}\DGnorm{\boldsymbol \varsigma_q^{n+1}+\boldsymbol \varsigma_q^{n}}{\boldsymbol v}-\frac{\varphi_q}{2}\DGnorm{\boldsymbol \chi^{n+1}+\boldsymbol \chi^n}{\boldsymbol v}
    =\tau_q\DGnorm{\boldsymbol E_q^n}{\boldsymbol v},\label{eq:vector:dgfem:error:s1:e2}
\end{align}
by Galerkin orthogonality, for any $\boldsymbol v\in\Dgvh$, where
\[\boldsymbol E_q(t)=\frac{\dot{\boldsymbol \psi}_q(t+\Delta t)+\dot{\boldsymbol \psi}_q(t)}{2}-\frac{\boldsymbol \psi_q(t+\Delta t)-\boldsymbol \psi_q(t)}{\Delta t}\textrm{ for each $q$}.\]
Inserting $\boldsymbol{v}=2(\boldsymbol{\chi}^{n+1}-\boldsymbol{\chi}^n)$ into \eqref{eq:vector:dgfem:error:s1:e1},
with \eqref{eq:vector:dgfem:dgr}, and $\boldsymbol{v}=\boldsymbol{\varsigma}^{n+1}_q-\boldsymbol{\varsigma}^n_q$ into \eqref{eq:vector:dgfem:error:s1:e2} for each $q$, summing over $n=0,\ldots,m-1$ for $m\leq N$ gives,
\begin{align*}
    {\rho}&\llnorm{\boldsymbol \varpi^{m}}^2
    +\DGnorm{\boldsymbol \chi^{m}}{\boldsymbol \chi^{m}}
    \nonumber+\sum_{q=1}^{N_\varphi}\frac{1}{\varphi_q}\DGnorm{\boldsymbol \varsigma_q^{m}}{\boldsymbol \varsigma_q^{m}}
    \\&\nonumber
    +{\Delta t}\sum_{n=0}^{m-1}\sum_{q=1}^{N_\varphi}\frac{2\tau_q}{\varphi_q}\DGnorm{\frac{\boldsymbol \varsigma_q^{n+1}-\boldsymbol \varsigma_q^n}{\Delta t}}{\frac{\boldsymbol \varsigma_q^{n+1}-\boldsymbol \varsigma_q^n}{\Delta t}}
    +\frac{\Delta t}{2}\sum_{n=0}^{m-1}\jump{\boldsymbol \varpi^{n+1}+\boldsymbol \varpi^n}{\boldsymbol \varpi^{n+1}+\boldsymbol \varpi^n}\nonumber
    \\\nonumber
=&{\rho}\llnorm{\boldsymbol \varpi^0}^2+\DGnorm{\boldsymbol \chi^0}{\boldsymbol \chi^{0}}+{\rho}\sum_{n=0}^{m-1}\Lnorm{\dot{\boldsymbol \theta}^{n+1}-\dot{\boldsymbol \theta}^n}{\boldsymbol \varpi^{n+1}+\boldsymbol \varpi^n}\\\nonumber&-{2\rho}\sum_{n=0}^{m-1}\Lnorm{\dot{\boldsymbol \theta}^{n+1}-\dot{\boldsymbol \theta}^n}{\boldsymbol {\mathcal{E}}_2^n}-{2\rho}\sum_{n=0}^{m-1}\Lnorm{\dot{\boldsymbol \theta}^{n+1}-\dot{\boldsymbol \theta}^n}{\boldsymbol {\mathcal{E}}_3^n}\\\nonumber&+{2\rho}\Delta t\sum_{n=0}^{m-1}\Lnorm{\boldsymbol {\mathcal{E}}_1^n}{\boldsymbol \varpi^{n+1}+\boldsymbol \varpi^n}-{\rho}\Delta t\sum_{n=0}^{m-1}\Lnorm{\boldsymbol {\mathcal{E}}_1^n}{\boldsymbol {\mathcal{E}}_2^n}-{2\rho}\Delta t\sum_{n=0}^{m-1}\Lnorm{\boldsymbol {\mathcal{E}}_1^n}{\boldsymbol {\mathcal{E}}_3^n}\\&+{2\rho}\sum_{n=0}^{m-1}\Lnorm{\boldsymbol \varpi^{n+1}-\boldsymbol \varpi^n}{\boldsymbol {\mathcal{E}}_2^n}+{2\rho}\sum_{n=0}^{m-1}\Lnorm{\boldsymbol \varpi^{n+1}-\boldsymbol \varpi^n}{\boldsymbol {\mathcal{E}}_3^n}
\nonumber\\
&
+2\sum_{q=1}^{N_\varphi}\DGnorm{\boldsymbol \chi^m}{\boldsymbol \varsigma_q^m}+\sum_{q=1}^{N_\varphi}\frac{2\tau_q}{\varphi_q}\DGnorm{\boldsymbol E_q^{m-1}}{\boldsymbol \varsigma_q^{m}}
-\sum_{n=0}^{m-1}\sum_{q=1}^{N_\varphi}\frac{2\tau_q}{\varphi_q}    \DGnorm{\boldsymbol E_q^{n+1}-\boldsymbol E_q^n}{\boldsymbol \varsigma_q^{n+1}},
\end{align*}
when we apply summation by parts to the summation terms from \eqref{eq:vector:dgfem:error:s1:e2}
with the fact $\boldsymbol{\varsigma}^0_q=\boldsymbol{0},\ \forall q$.
Hence the definition of DG bilinear form and its coercivity imply that
\begin{align}
    {\rho}&\llnorm{\boldsymbol \varpi^{m}}^2
    +\enorm{\boldsymbol \chi^{m}}^2
    \nonumber+\sum_{q=1}^{N_\varphi}\frac{1}{\varphi_q}\enorm{\boldsymbol \varsigma_q^{m}}^2+{\Delta t}\sum_{n=0}^{m-1}\sum_{q=1}^{N_\varphi}\frac{2\kappa\tau_q}{\varphi_q}\left\lVert{\frac{\boldsymbol \varsigma_q^{n+1}-\boldsymbol \varsigma_q^n}{\Delta t}}\right\rVert^2\\&+\frac{\Delta t}{2}\sum_{n=0}^{m-1}\jump{\boldsymbol \varpi^{n+1}+\boldsymbol \varpi^n}{\boldsymbol \varpi^{n+1}+\boldsymbol \varpi^n}\nonumber
\\
\nonumber
\leq&
\bigg|{\rho}\llnorm{\boldsymbol \varpi^0}^2+\enorm{\boldsymbol \chi^0}^2+{\rho}\sum_{n=0}^{m-1}\Lnorm{\dot{\boldsymbol \theta}^{n+1}-\dot{\boldsymbol \theta}^n}{\boldsymbol \varpi^{n+1}+\boldsymbol \varpi^n}
\\
\nonumber&
-{2\rho}\sum_{n=0}^{m-1}\Lnorm{\dot{\boldsymbol \theta}^{n+1}-\dot{\boldsymbol \theta}^n}{\boldsymbol {\mathcal{E}}_2^n}-{2\rho}\sum_{n=0}^{m-1}\Lnorm{\dot{\boldsymbol \theta}^{n+1}-\dot{\boldsymbol \theta}^n}{\boldsymbol {\mathcal{E}}_3^n}
\\
\nonumber&
+{\rho}\Delta t\sum_{n=0}^{m-1}\Lnorm{\boldsymbol {\mathcal{E}}_1^n}{\boldsymbol \varpi^{n+1}+\boldsymbol \varpi^n}-{2\rho}\Delta t\sum_{n=0}^{m-1}\Lnorm{\boldsymbol {\mathcal{E}}_1^n}{\boldsymbol {\mathcal{E}}_2^n}
-{2\rho}\Delta t\sum_{n=0}^{m-1}\Lnorm{\boldsymbol {\mathcal{E}}_1^n}{\boldsymbol {\mathcal{E}}_3^n}
\\&
+{2\rho}\sum_{n=0}^{m-1}\Lnorm{\boldsymbol \varpi^{n+1}-\boldsymbol \varpi^n}{\boldsymbol {\mathcal{E}}_2^n}+{2\rho}\sum_{n=0}^{m-1}\Lnorm{\boldsymbol \varpi^{n+1}-\boldsymbol \varpi^n}{\boldsymbol {\mathcal{E}}_3^n}
\nonumber\\
&\nonumber
+2\sum_{q=1}^{N_\varphi}\DGnorm{\boldsymbol \chi^m}{\boldsymbol \varsigma_q^m}+\sum_{q=1}^{N_\varphi}\frac{2\tau_q}{\varphi_q}\DGnorm{\boldsymbol E_q^{m-1}}{\boldsymbol \varsigma_q^{m}}
-\sum_{n=0}^{m-1}\sum_{q=1}^{N_\varphi}\frac{2\tau_q}{\varphi_q}    
\DGnorm{\boldsymbol E_q^{n+1}-\boldsymbol E_q^n}{\boldsymbol \varsigma_q^{n+1}}
\nonumber\\&
+2\sume\{\ushort{\boldsymbol D}\ushort{\boldsymbol \varepsilon}(\boldsymbol  {\chi}^m)\}:\njump{\boldsymbol{\chi}^m}de
+\sumq\frac{2}{\varphi_q}\sume \{\ushort{\boldsymbol D}\ushort{\boldsymbol \varepsilon}(\boldsymbol  {\varsigma}^m_q)\}:\njump{\boldsymbol{\varsigma}^m_q}de\bigg|.
\label{eq:vector:dgfem:error:s1:e3}
\end{align}
Using the elliptic projection, \eqref{eq:vector:dgfem:s1e3}, \eqref{eq:vector:dgfem:s1e4}
and initial conditions, we have $\llnorm{\boldsymbol \varpi^0}\leq\llnorm{\dot{\boldsymbol \theta}^0}$
and $\enorm{\boldsymbol \chi^0}=0$.
Further, we note that the fundamental theorem of calculus allows us to deal with the time differences 
with expressions like
$\dot{\boldsymbol \theta}^{n+1}-\dot{\boldsymbol \theta}^n=\int^{t_{n+1}}_{t_n}\ddot{\boldsymbol{\theta}}(t')dt'$,
and then using \eqref{eq:error:CN} we can bound
$\boldsymbol{\mathcal{E}}_1^n,\boldsymbol{\mathcal{E}}_2^n,\boldsymbol{\mathcal{E}}_3^n$
and $\boldsymbol{E}_q^n$, for all $n$ and for each $q$, in an optimal way. For example, 
\begin{align*}
    \bigg|\Delta t\sum_{n=0}^{m-1}\Lnorm{\boldsymbol {\mathcal{E}}_1^n}{\boldsymbol {\mathcal{E}}_2^n}\bigg|\leq&\frac{\Delta t}{2}\sum_{n=0}^{N-1}\llnorm{\boldsymbol {\mathcal{E}}_1^n}^2+\frac{\Delta t}{2}\sum_{n=0}^{N-1}\llnorm{\boldsymbol {\mathcal{E}}_2^n}^2\\\leq&\frac{\Delta t^4}{8}\Llnorm{\boldsymbol {u}^{(4)}}^2+\frac{\Delta t^4}{8}\Llnorm{\boldsymbol {\theta}^{(3)}}^2,
\end{align*}
by the Cauchy-Schwarz inequality, Young's inequality and \eqref{eq:error:CN}.
On the other hand, the elliptic error estimates such as \eqref{vector:appl2} provide spatial error estimates for
$\boldsymbol{\theta}(t)$ and its time derivatives and, noting that the internal variables
are analogues of displacement, $\boldsymbol{u}$, we can employ elliptic error estimates
for the internal variables as well. Therefore, in the same way as for the stability estimate, using the Cauchy-Schwarz and
Young's inequalities, \eqref{eq:vector:dg:bdd2}, summation by parts, maximising in time, and choosing large enough
penalty parameters, we eventually arrive at,
\begin{align*}
        \max_{1\leq n\leq N}&\llnorm{\boldsymbol\varpi^{n}}
    +\max_{1\leq n\leq N}\enorm{\boldsymbol\chi^{n}}+
    \sum_{q=1}^{N_\varphi}\max_{1\leq n\leq N}\enorm{\boldsymbol\varsigma_q^{n}}+\left({{\Delta t}\sum_{n=0}^{N-1}\sum_{q=1}^{N_\varphi}\frac{\tau_q}{\varphi_q}\left\lVert{\frac{\boldsymbol\varsigma_q^{n+1}-\boldsymbol\varsigma_q^n}{\Delta t}}\right\rVert^2}\right)^{1/2}\\&+\left({{\Delta t}\sum_{n=0}^{N-1}\jump{\boldsymbol\varpi^{n+1}+\boldsymbol\varpi^{n}}{\boldsymbol\varpi^{n+1}+\boldsymbol\varpi^{n}}}\right)^{1/2}
\leq
CT\lVert \boldsymbol u\rVert_{H^4(0,T;\Dgv)}(h^{r}+\Delta t^2).
\end{align*}
Here $C$ is independent of $h$, $\Delta t$ and the solutions, but depends on $\rho$, $\Omega$, $\partial\Omega$ and the 
material properties. For full technical details for the proof, we refer to \cite{jang2020spatially}.
\end{proof}

Next, the analogous estimate for {$\textbf{(V)}^h$}.

\begin{lem}
\label{lemma:vector:dgfem:error:s2}
Suppose
$\boldsymbol u\in H^2(0,T;\boldsymbol C^2({\Omega}))\cap W^1_\infty(0,T;\Dgv)\cap H^4(0,T;\Dgv)$
and $\beta_0(d-1)\geq 1$ for $s>3/2$. If, for $n=0,1,\ldots,N$,
the solution to \textnormal{$\textbf{(V)}^h$} is
$\boldsymbol{U}^n_h$, $\boldsymbol{W}^n_h$, $\boldsymbol{\mathcal{S}}^n_{h1},\ldots,\boldsymbol{\mathcal{S}}^n_{h\Nphi}$ then, 
for large enough $\alpha_0$, there exists a positive constant $C$ such that
\begin{gather*}
    \max_{1\leq n\leq N}\llnorm{\boldsymbol \varpi^{n}}+\max_{1\leq n\leq N}\enorm{\boldsymbol \chi^{n}}\leq CT\lVert\boldsymbol{u}\rVert_{H^4(0,T;\dgv)}(h^{r}+\Delta t^2),
\end{gather*}
where $r=\min{(k+1,s)}$ and $C$ are independent of $h$, $\Delta t$, $T$ and the numerical solution.
\end{lem}

\begin{proof}
We follow similar steps as in the proof of Lemma \ref{lemma:vector:dgfem:error:s1} for the displacement form
to show this claim. By averaging at $t=t_{n+1}$ and $t=t_n$, subtracting \eqref{eq:vector:dgfem:s2:primal1} from \eqref{eq:vector:dgfem:s2e1} with $\boldsymbol{v}=2(\boldsymbol{\chi}^{n+1}-\boldsymbol{\chi}^n)$, and subtracting
\eqref{eq:vector:dgfem:s2:primal2} from \eqref{eq:vector:dgfem:s2e2}, with
$\boldsymbol{v}=2(\boldsymbol{\Upsilon}^{n+1}_q+\boldsymbol{\Upsilon}^{n}_q)$ for each $n$ and $q$,
summing over $n=0,\ldots,m-1$ for $m\leq N$, employing summation by parts, recalling the coercivity
of $\DGnorm{\cdot}{\cdot}$, and using Galerkin orthogonality, we eventually get,
\begin{align}
{\rho}&\llnorm{\boldsymbol \varpi^{m}}^2+{\varphi_0}\enorm{\boldsymbol \chi^{m}}^2\nonumber+\sum_{q=1}^{N_\varphi}\frac{\kappa}{\varphi_q }\enorm{\boldsymbol \Upsilon_q^{m}}^2\nonumber+{\Delta t}\sum_{n=0}^{m-1}\sum_{q=1}^{N_\varphi}\frac{\kappa}{\tau_q\varphi_q}\enorm{\boldsymbol \Upsilon_q^{n+1}+\boldsymbol \Upsilon_q^n}^2\nonumber\\
&+\frac{\Delta t}{2}\sum_{n=0}^{m-1}\jump{\boldsymbol \varpi^{n+1}+\boldsymbol \varpi^n}{\boldsymbol \varpi^{n+1}+\boldsymbol \varpi^n}\nonumber\\
\leq&\bigg|{\rho}\llnorm{\boldsymbol \varpi^0}^2+{\varphi_0}\enorm{\boldsymbol \chi^0}^2\nonumber+{\rho}\sum_{n=0}^{m-1}\Lnorm{\dot{\boldsymbol \theta}^{n+1}-\dot{\boldsymbol \theta}^n}{\boldsymbol \varpi^{n+1}+\boldsymbol \varpi^n}\\&+{\rho}{\Delta t}\sum_{n=0}^{m-1}\Lnorm{\boldsymbol {\mathcal{E}}_1^n}{\boldsymbol \varpi^{n+1}+\boldsymbol \varpi^n}\nonumber-2{\rho}\sum_{n=0}^{m-1}\Lnorm{\dot{\boldsymbol \theta}^{n+1}-\dot{\boldsymbol \theta}^n}{\boldsymbol {\mathcal{E}}_2^n}\\&-2{\rho}\sum_{n=0}^{m-1}\Lnorm{\dot{\boldsymbol \theta}^{n+1}-\dot{\boldsymbol \theta}^n}{\boldsymbol {\mathcal{E}}_3^n}-2{\rho}{\Delta t}\sum_{n=0}^{m-1}\Lnorm{\boldsymbol {\mathcal{E}}_1^n}{\boldsymbol {\mathcal{E}}_2^n}\nonumber-2{\rho}{\Delta t}\sum_{n=0}^{m-1}\Lnorm{\boldsymbol {\mathcal{E}}_1^n}{\boldsymbol {\mathcal{E}}_3^n}\\&+{\Delta t}\sum_{n=0}^{m-1}\sum_{q=1}^{N_\varphi}\frac{1}{\varphi_q}\DGnorm{\boldsymbol E^n_q}{\boldsymbol \Upsilon_q^{n+1}+\boldsymbol \Upsilon_q^n}-{\Delta t}\sum_{n=0}^{m-1}\sum_{q=1}^{N_\varphi}\DGnorm{\boldsymbol {\mathcal{E}}_3^n}{\boldsymbol \Upsilon_q^{n+1}+\boldsymbol \Upsilon_q^n}\nonumber\\
&+2\varphi_0\sume\{\ushort{\boldsymbol D}\ushort{\boldsymbol \varepsilon}(\boldsymbol  {\chi}^m)\}:\njump{\boldsymbol{\chi}^m}de\bigg|,\label{pf:s2:error:eq2}\end{align}
where
\[
\boldsymbol E_q(t):=\frac{\dot{\boldsymbol \zeta}_q(t+\Delta t)+\dot{\boldsymbol \zeta}_q(t)}{2}-\frac{\boldsymbol \zeta_q(t+\Delta t)-\boldsymbol \zeta_q(t)}{\Delta t}\textrm{ for each }q.
\]
We complete the proof by using the same techniques and results as used earlier: the Cauchy-Schwarz
and Young's inequalities, integration and summation by parts, \eqref{eq:vector:dg:bdd2}, \eqref{vector:appl2} and
\eqref{eq:error:CN}.
\end{proof}

We can now use Lemmas~\ref{lemma:vector:dgfem:error:s1} and~\ref{lemma:vector:dgfem:error:s2} to
prove the following \textit{a priori} error estimates.

\begin{thm}[Error analysis]\label{thm:vector:dgfem:fully:errorestimates}
Suppose the discrete solutions in $\Dgvh$ satisfy either \textnormal{$\textbf{(D)}^h$} or
\textnormal{$\textbf{(V)}^h$} for integer $s\geq2,$ and also that
$\boldsymbol u\in H^2(0,T;\boldsymbol C^2({\Omega}))\cap W^1_\infty(0,T;\Dgv)\cap H^4(0,T;\Dgv)$.
If Lemmas~\textnormal{\ref{lemma:vector:dgfem:error:s1}}
and~\textnormal{\ref{lemma:vector:dgfem:error:s2}} hold, then we have
\begin{alignat*}{2}
&\max_{1\leq n\leq N}\enorm{\boldsymbol u(t_n)-\boldsymbol U^n_h}
&& \leq
CT\lVert\boldsymbol{u}\rVert_{H^4(0,T;\dgv)}(h^{r-1}+\Delta t^2),
\\
&\max_{1\leq n\leq N}\llnorm{\boldsymbol u(t_n)-\boldsymbol U^n_h}
&& \leq
CT\lVert\boldsymbol{u}\rVert_{H^4(0,T;\dgv)}(h^{r}+\Delta t^2),
\\
&\max_{1\leq n\leq N}\enorm{\dot{\boldsymbol u}(t_n)-\boldsymbol W^n_h}
&& \leq
CT\lVert\boldsymbol{u}\rVert_{H^4(0,T;\dgv)}(h^{r-1}+\Delta t^2),
\\
\text{and }\qquad
&\max_{1\leq n\leq N}\llnorm{\dot{\boldsymbol u}(t_n)-\boldsymbol W^n_h}
&& \leq CT\lVert\boldsymbol{u}\rVert_{H^4(0,T;\dgv)}(h^{r}+\Delta t^2),
\end{alignat*}
where $r=\min(s+1,k)$ and $C$ is a positive constant that is independent of $h$, $\Delta t$, $T$
and the exact and discrete solutions, but depends on $\Omega$, $\partial \Omega$, and the
material coefficients. Moreover, the initial discrete errors are given by
\[\enorm{\boldsymbol u_0-\boldsymbol U^0_h}\leq
Ch^{r-1}\vertiii{\boldsymbol{u}_0}_{\dgv},\qquad\text{and}\qquad\llnorm{\boldsymbol w_0-\boldsymbol W^0_h}\leq
Ch^{r}\vertiii{\boldsymbol{w}_0}_{\dgv}.
\]
\end{thm}

\begin{proof}
Using the triangle inequality with Lemmas~\ref{lemma:vector:dgfem:error:s1}
and~\ref{lemma:vector:dgfem:error:s2}, and the DG elliptic error estimates, we get,
for any $n\in\{1,\ldots,N\}$, that
\begin{align}
\enorm{\boldsymbol u(t_n)-\boldsymbol U^n_h}\leq&\enorm{\boldsymbol \theta(t_n)}+\enorm{\boldsymbol  \chi^n}\leq CT\lVert\boldsymbol{u}\rVert_{H^4(0,T;\dgv)}(h^{r-1}+\Delta t^2).\label{pf:thm:error:eq1}
\end{align}
Recalling Poincar\'e's inequality \eqref{dgpoincare} indicates that
$\llnorm{\boldsymbol{v}}\leq C\enorm{\boldsymbol{v}}$ for any ${\boldsymbol{v}}\in\Dgvh$. Hence we have
\begin{align}
\llnorm{\boldsymbol u(t_n)-\boldsymbol U^n_h}
\leq\llnorm{\boldsymbol \theta(t_n)}+C\enorm{\boldsymbol  \chi^n}
\leq 
CT\lVert\boldsymbol{u}\rVert_{H^4(0,T;\dgv)}(h^{r-1}+\Delta t^2).
\label{pf:thm:error:eq2}
\end{align}
In a similar way, we have
\begin{align}
\llnorm{\dot{\boldsymbol u}(t_n)-\boldsymbol W^n_h}
\leq&
\llnorm{\dot{\boldsymbol \theta}(t_n)}+\llnorm{\boldsymbol  \varpi^n}
\leq CT\lVert\boldsymbol{u}\rVert_{H^4(0,T;\dgv)}(h^{r}+\Delta t^2).
\label{pf:thm:error:eq3}
\end{align}
On the other hand, as seen in \eqref{eq:vector:dgfem:error:s1:e3} and \eqref{pf:s2:error:eq2},
the jump penalty term of $\boldsymbol{\varpi}^n$ is also bounded by the same upper bound. In other words, we have
\[
\left(\jump{\boldsymbol{\varpi}^n}{\boldsymbol{\varpi}^n}\right)^{1/2}=O(h^r+\Delta t^2).
\]
Using this argument with the inverse inequality \eqref{dginverse}, we can derive energy norm error estimates
for the velocity vector field and get,
\begin{align}
    \enorm{\dot{\boldsymbol u}(t_n)-\boldsymbol W^n_h}
    \leq&
    \enorm{\dot{\boldsymbol \theta}(t_n)}+\enorm{\boldsymbol  \varpi^n}
    \leq
    \enorm{\dot{\boldsymbol \theta}(t_n)}+Ch^{-1}\llnorm{\boldsymbol  \varpi^n}
    +\left(\jump{\boldsymbol{\varpi}^n}{\boldsymbol{\varpi}^n}\right)^{1/2}
    \nonumber\\
    \leq&
    CT\lVert\boldsymbol{u}\rVert_{H^4(0,T;\dgv)}(h^{r-1}+\Delta t^2)
    \label{pf:thm:error:eq4}.
\end{align}
Since $n$ is arbitrary and the the right hand sides of \eqref{pf:thm:error:eq1}-\eqref{pf:thm:error:eq4} are independent of $n$, the proof is completed. Furthermore, to show the discrete errors for $n=0$, we want to use the triangular inequalities, DG elliptic error estimates and initial (DG elliptic/$L_2$) projections, e.g. \eqref{eq:vector:dgfem:s1e3} and \eqref{eq:vector:dgfem:s1e4}. Hence we have
\begin{align*}
    \enorm{\boldsymbol u_0-\boldsymbol U^0_h}=\enorm{\boldsymbol \theta(0)-\boldsymbol \chi^0}\leq\enorm{\boldsymbol \theta(0)}+\enorm{\boldsymbol \chi^0}=\enorm{\boldsymbol \theta(0)}\leq
Ch^{r-1}\vertiii{\boldsymbol{u}_0}_{\dgv},
\end{align*}by the fact $\enorm{\boldsymbol \chi^0}=0$ and \eqref{vector:appDG}. Also, since $\llnorm{\boldsymbol \varpi^0}\leq\llnorm{\dot{\boldsymbol \theta}(0)}$, \eqref{vector:appl2} leads to
\begin{align*}
    \llnorm{\boldsymbol w_0-\boldsymbol W^0_h}=\llnorm{\dot{\boldsymbol \theta}(0)-\boldsymbol \varpi^0}\leq\llnorm{\dot{\boldsymbol \theta}(0)}+\llnorm{\boldsymbol \varpi^0}\leq2\llnorm{\dot{\boldsymbol \theta}(0)}\leq
Ch^{r}\vertiii{\boldsymbol{w}_0}_{\dgv}.
\end{align*}
\end{proof}

\section{Numerical Experiments}\label{sec:numerical experiments}
In this section we use the FEniCS environment, see \cite{alnaes2015fenics} and
\url{https://fenicsproject.org}, to give the results of some numerical experiments
that were carried out to demonstrate the convergence rates proven above:
we present tables of numerical errors, as well as convergence rates.
The python codes are available at Jang's Github
(\url{https://github.com/Yongseok7717/Visco_Int_CG}),
and can be used to reproduce the results that are given below.
Alternatively, these results can also be reproduced by using docker. In this case 
the commands to run are:
\begin{verbatim}
docker pull jangyour/fenics_fem
docker run -ti jangyour/fenics_fem
cd  visco/Int_DG; . main.sh
\end{verbatim}

\paragraph*{An exact solution}
As is common, we use an artificial, or \emph{manufactured}, exact solution in order to
calculate the errors. Let this manufactured strong solution be
\[\boldsymbol u(x,y,t)=\left[\begin{array}{c}
      xye^{1-t}\\
      \cos(t)\sin(xy)
\end{array}\right]\in C^{\infty}(0,T;\boldsymbol C^{\infty}(\Omega))\]where $\Omega=[0,1]\times[0,1]$ and $T=1$. We use two internal variables with coefficients and time delays
$\varphi_0=0.5,\ \varphi_1=0.1,\ \varphi_2=0.4,\ \tau_1=0.5,\ \tau_2=1.5$ and (because
we are not constrained by the physical reality here) we assume for simplicity an
identity fourth order tensor as our $\ushort{\boldsymbol{D}}$
so that $\ushort{\boldsymbol{D}}\ushort{\boldsymbol{\varepsilon}}=\ushort{\boldsymbol{\varepsilon}}$.
We define the Dirichlet boundary as
\[
\Gamma_D=\{(x,y)\in\partial\Omega\ |\ x=0\textrm{ or }y=0\}
\]
and then $\boldsymbol{u}=\boldsymbol{0}$ on $\Gamma_D$, and the other data are readily computed.

\paragraph*{Numerical results}
We present numerical errors to demonstrate the convergence rates given by the theorems and for this
we define $\boldsymbol{e}^n=\boldsymbol{u}(t_n)-\boldsymbol{U}_h^n$ 
and $\tilde{\boldsymbol{e}}^n=\boldsymbol{w}(t_n)-\boldsymbol{W}_h^n$ for $n=0,\ldots,N$. According to Theorem \ref{thm:vector:dgfem:fully:errorestimates}, we have the $L_2$ norm error as well as DG energy norm error such that
\[
\enorm{\boldsymbol{e}^n}=O(h^{k}+\Delta t^{2}),\text{ and }\llnorm{{\boldsymbol{e}}^n}=O(h^{k+1}+\Delta t^{2}),
\]
since $s=\infty$. Also, we have the same rates of convergence for the velocity error.
By recalling Korn's inequality \eqref{eq:kordDG}, we note that the broken $H^1$ norm is bounded by the energy norm.
Therefore, we have broken $H^1$ error estimates such that
\[
\hsnorm{\boldsymbol{e}^n}+\hsnorm{\tilde{\boldsymbol{e}}^n}\leq C(h^{k}+\Delta t^{2}),
\]
for some positive constant $C$. Since the DG energy norm depends on the penalty parameters
$\alpha_0$ and $\beta_0$, we hereafter consider the errors in the broken $H^1$ norm instead.

Firstly, we set $\Delta t=h$ and define the numerical convergence rate $d_c$ by 
\[
d_c\ =\ \frac{\log(\text{error of }h_1)-\log(\text{error of }h_2)}{\log(h_1)-\log(h_2)}.
\]
We can see in Figure \ref{fig:convergentOrder} that $d_c$ is 1 for both the displacement and velocity form errors when $k=1$,
and that it suggests quadratic convergence orders for $L_2$ norm error or when $k=2$, regardless of which internal variable form
is used.

\begin{figure}[ht]
	\centering
	\includegraphics[width=0.65\linewidth]{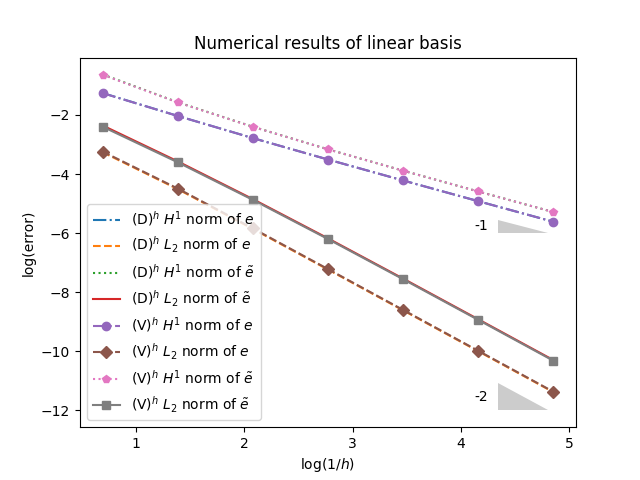}
	\includegraphics[width=0.65\linewidth]{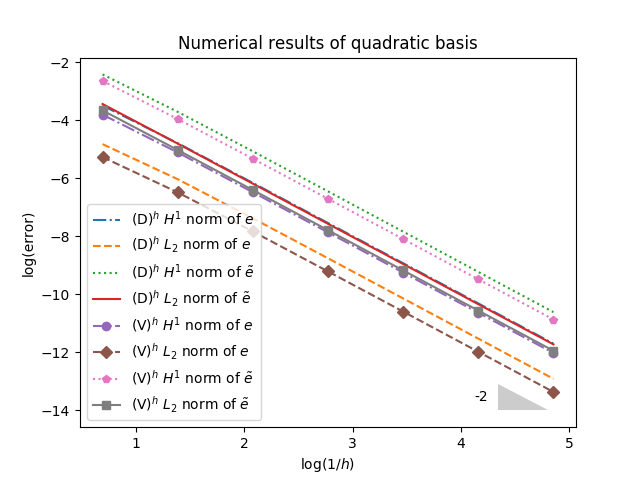}
	\caption{Numerical convergence order: \textbf{linear} (top) and \textbf{quadratic} (bottom) polynomial basis}
	\label{fig:convergentOrder}
\end{figure}

Secondly, if we take $\Delta t \ll h$ we can render the `time error' negligible in comparison to
the `space error' and isolate the spatial convergence rates.
Some example results for this are shown in Tables \ref{table:fixedTime:H1} and \ref{table:fixedTime:L2},
where we see that the numerical rates follow the spatial convergence order of $d_c=k$ (resp. $d_c=k+1$) in the $H^1$ (resp. $L_2$) norm.
We can also see that there is no significant difference between the numerical schemes $\textbf{(D)}^h$ and
$\textbf{(V)}^h$, in terms of convergence rates as well as the size of the errors.

Thirdly, in Table \ref{table:fixed_h} we give some results for when the spatial error became negligibly small to the temporal error, to observe second order accuracy in time regardless of spatial norms, displacement/velocity fields and internal variable forms.
All of these results are consistent with the claims in the theorems.

\begin{table}[H]
	\centering
	\begin{tabular}{cc|cc|cc}
			\multirow{2}{*}{$k$}&\multirow{2}{*}{$h$}&\multicolumn{2}{c|}{\textbf{Displacement form}}&\multicolumn{2}{c}{\textbf{Velocity form}}\\		
		&&$\hsnorm{ e^N}$&$\hsnorm{\tilde e^N}$ &$\hsnorm{ e^N}$ &$\hsnorm{\tilde e^N}$ \\\midrule
		\multirow{4}{*}{1}
		&1/4&1.298e-01&1.951e-01&1.298e-01&1.951e-01\\
		&1/8&6.177e-02 (1.07)&8.741e-02 (1.16)&6.177e-02 (1.07)&8.741e-02 (1.16)\\
		&1/16&2.993e-02 (1.05)&4.130e-02 (1.08)&2.993e-02 (1.05)&4.130e-02 (1.08)\\
		&1/32&1.473e-02 (1.02)&2.001e-02 (1.04)&1.473e-02 (1.02)&2.001e-02 (1.04)\\\hline
	\multirow{4}{*}{2}
		&1/4&3.168e-03&4.996e-03&3.168e-03&4.996e-03\\
		&1/8&8.030e-04 (1.98)&1.284e-03 (1.96)&8.030e-04 (1.98)&1.284e-03 (1.96)\\
		&1/16&2.008e-04 (2.00)&3.256e-04 (1.98)&2.008e-04 (2.00)&3.256e-04 (1.98)\\
		&1/32&5.010e-05 (2.00)&8.206e-05 (1.99)&5.010e-05 (2.00)&8.206e-05 (1.99)\\
	\end{tabular}
	\caption{$H^1$ norm of numerical errors (orders) when $\Delta t=1/2048$}\label{table:fixedTime:H1}
\end{table} \vspace{-1.2cm}
\begin{table}[H]
	\centering
	\begin{tabular}{cc|cc|cc}
			\multirow{2}{*}{$k$}&\multirow{2}{*}{$h$}&\multicolumn{2}{c|}{\textbf{Displacement form}}&\multicolumn{2}{c}{\textbf{Velocity form}}\\		
		&&$\llnorm{ e^N}$&$\llnorm{\tilde e^N}$ &$\llnorm{ e^N}$ &$\llnorm{\tilde e^N}$ \\\midrule
		\multirow{4}{*}{1}
		&1/4&1.067e-02&2.293e-02&1.067e-02&2.293e-02\\
		&1/8&2.808e-03 (1.93)&6.691e-03 (1.78)&2.808e-03 (1.93)&6.691e-03 (1.78)\\
		&1/16&7.094e-04 (1.98)&1.182e-03 (1.88)&7.094e-04 (1.98)&1.182e-03 (1.88)\\
		&1/32&1.781e-04 (1.99)&4.686e-04 (1.95)&1.781e-04 (1.99)&4.686e-04 (1.95)\\\hline
	\multirow{4}{*}{2}
		&1/4&8.362e-05&1.496e-04&8.362e-05&1.496e-04\\
		&1/8&1.011e-05 (3.05)&1.861e-05 (3.01)&1.011e-05 (3.05)&1.861e-05 (3.01)\\
		&1/16&1.231e-06 (3.04)&2.315e-06 (3.01)&1.231e-06 (3.04)&2.315e-06 (3.01)\\
		&1/32&1.515e-07 (3.02)&2.906e-07 (3.00)&1.514e-07 (3.02)&2.902e-07 (3.00)\\
	\end{tabular}
	\caption{$L_2$ norm of numerical errors (orders) when $\Delta t=1/2048$}\label{table:fixedTime:L2}
\end{table} \vspace{-1.2cm}
\begin{table}[H]
	\centering
	\begin{tabular}{cc|cc|cc}
			\multirow{2}{*}{Norm\kern-0.8em}&\multirow{2}{*}{$\Delta t$\kern-0.5em}&\multicolumn{2}{c|}{\textbf{Displacement form}}&\multicolumn{2}{c}{\textbf{Velocity form}}\\		
		&&${ e^N}$&${\tilde e^N}$ &${ e^N}$ &${\tilde e^N}$ \\\midrule
		\multirow{4}{*}{$H^1$}&1/2&2.696e-02&9.579e-02&1.766e-02&7.348e-02\\
		&1/4&7.445e-03 (1.86)&2.461e-02 (1.96)&4.879e-03 (1.86)&1.880e-02 (1.97)\\
		&1/8&1.894e-03 (1.97)&6.169e-03 (2.00)&1.2429e-03 (1.97)&4.712e-03 (2.00)\\
		&1/16&4.747e-04 (2.00)&1.549e-04 (2.00)&3.117e-04 (2.00)&1.181e-03 (2.00)\\\hline
		\multirow{4}{*}{$L_2$}&1/2&8.121e-02&3.222e-02&5.256e-03&2.586e-02\\
		&1/4&2.410e-03 (1.75)&8.221e-03 (1.97)&1.534e-03 (1.78)&6.601e-03 (1.97)\\
		&1/8&6.252e-04 (1.95)&2.067e-03 (1.99)&3.974e-04 (1.95)&1.659e-03 (1.99)\\
		&1/16&1.577e-04 (1.99)&5.178e-04 (2.00)&1.001e-04 (1.99)&4.155e-04 (2.00)\\
	\end{tabular}
	\caption{Time convergence errors (orders) for fixed $h=1/128$ and $k=2$}\label{table:fixed_h}
\end{table}

\paragraph*{Penalty parameters}
For well-posedness as well as optimal convergence, we need to have `large enough' penalty parameters and following
studies of elasticity problems, \cite{houston2006hp,wihler2006locking}, we took $\alpha_0=10$ and $\beta_0=1$ for
the computations above. If $\alpha_0$ is small our numerical schemes will lose stability and convergence. 
This is illustrated in Table \ref{table:smallalpha} where we see that the numerical errors diverged
when $\alpha_0=0.1$. On the other hand, in order to observe the requirements of the stability and error analyses above
we need $\beta_0(d-1)\geq 1$ but, a too large $\beta_0$ will result in the assembled global matrix being
ill-conditioned. Therefore, taking $\beta_0=1/(d-1)$ seems to be the most reasonable choice.
In fact, the condition number of the assembled system matrix in the interior penalty method is of order $O(h^{-(\beta_0+1)})$ \cite{jang2020spatially,DG}, and hence we may
encounter difficulty in solving linear systems for large $\beta_0$ if we use iterative solvers.
For more details, we refer to \cite{jang2020spatially,DG}.

\begin{table}[H]
	\centering
	\begin{tabular}{c|ccc|ccc}
	&\multicolumn{3}{c|}{\textbf{Displacement form}}&\multicolumn{3}{c}{\textbf{Velocity form}}\\
	$h$&$1/2$&$1/4$&$1/8$&$1/2$&$1/4$&$1/8$\\
		\midrule
		Error&2.262&9.527e+03&1.262e+14&2.286&9.364e+03&1.266e+14\\
	\end{tabular}
	\caption{$\llnorm{ e^N}$ for $\alpha_0=0.1$ where $\Delta t= h$ and $k=1$}
	\label{table:smallalpha}
\end{table}

\section{Conclusion and Discussion}\label{sec:conclusion}
We have presented an \textit{a priori} stability and error analysis of the DGFEM for
linear dynamic viscoelasticity with two types of internal variables used to replace
the fading memory Volterra integral. We have given illustrative numerical results, 
explained how these results can be reproduced, and we have observed that both of the schemes
(for each type of internal variable) behaves similarly. The numerics are
consistent with the predicted convergence rates.

The main achievement of this work is that we have been able to give the stability
and error bounds without recourse to Gr\"onwall's inequality, with the result that the
stability and error bounds depend explicitly on the final time linearly, and not exponentially,
(along with the time dependence in the exact solution norms).
This is a significant improvement over the estimates given in \cite{DGV}.
Moreover, the schemes presented here use displacement or velocity
internal variables rather than stresses, and so --- compared to \cite{DGV} --- for the schemes
presented here, there is a modest reduction in the computer memory requirements. 

We assumed throughout that $\varphi_0>0$ in \eqref{eq:stressrelax:exponential}, which is 
reasonable for materials with long term stiffness (`solids' as defined by Golden and Graham
in \cite{golden2013boundary}).
However, for `fluids', under the specific condition $\varphi_0=0$, the velocity form will yield an
interesting alternative model.
In particular, with $\boldsymbol{u}_0=\boldsymbol{0}$ for  simplicity,
we can rewrite \eqref{primal:internalVelo} as  
\[
\rho \dot{\boldsymbol w}-\sum_{q=1}^{N_\varphi}\nabla\cdot\left(\ushort{{\boldsymbol {D}}}\ushort{{\boldsymbol\varepsilon}}\big(\boldsymbol{\zeta}_q\big)\right)=\boldsymbol f,
\]
where we have set $\dot{\boldsymbol{u}}=\boldsymbol{w}$. The internal variable still evolves according to the
ODE, \eqref{primal:ode:velo}, which holds at every point in space, and so in the numerical scheme there is no need
to introduce displacement. This results in a reduced memory requirement.
We plan to explore this interesting variant in a future study.

\bibliographystyle{acm}
\bibliography{ref}

\begin{thebibliography}{10}

\bibitem{alnaes2015fenics}
{\sc Aln{\ae}s, M., Blechta, J., Hake, J., Johansson, A., Kehlet, B., Logg, A.,
  Richardson, C., Ring, J., Rognes, M.~E., and Wells, G.~N.}
\newblock {The FEniCS project version 1.5}.
\newblock {\em Archive of Numerical Software 3}, 100 (2015), 9--23.

\bibitem{brenner2003poincare}
{\sc Brenner, S.~C.}
\newblock {Poincar{\'e}--Friedrichs inequalities for piecewise $H^1$
  functions}.
\newblock {\em SIAM Journal on Numerical Analysis 41}, 1 (2003), 306--324.

\bibitem{brenner2004korn}
{\sc Brenner, S.~C.}
\newblock {Korn's inequalities for piecewise $H^1$ vector fields}.
\newblock {\em Mathematics of Computation\/} (2004), 1067--1087.

\bibitem{Crank1996}
{\sc Crank, J., and Nicolson, P.}
\newblock A practical method for numerical evaluation of solutions of partial
  differential equations of the heat-conduction type.
\newblock {\em Advances in Computational Mathematics 6}, 1 (Dec 1996),
  207--226.

\bibitem{drozdov1998viscoelastic}
{\sc Drozdov, A.~D.}
\newblock {\em Viscoelastic structures: mechanics of growth and aging}.
\newblock Academic Press, 1998.

\bibitem{findley2013creep}
{\sc Findley, W.~N., and Davis, F.~A.}
\newblock {\em Creep and relaxation of nonlinear viscoelastic materials}.
\newblock Courier Corporation, 2013.

\bibitem{golden2013boundary}
{\sc Golden, J.~M., and Graham, G.~A.}
\newblock {\em Boundary value problems in linear viscoelasticity}.
\newblock Springer Science \& Business Media, 2013.

\bibitem{hansbo2002discontinuous}
{\sc Hansbo, P., and Larson, M.~G.}
\newblock {Discontinuous Galerkin methods for incompressible and nearly
  incompressible elasticity by Nitsche's method}.
\newblock {\em Computer methods in applied mechanics and engineering 191},
  17-18 (2002), 1895--1908.

\bibitem{houston2006hp}
{\sc Houston, P., Sch{\"o}tzau, D., and Wihler, T.~P.}
\newblock {An $hp$-adaptive mixed discontinuous Galerkin FEM for nearly
  incompressible linear elasticity}.
\newblock {\em Computer methods in applied mechanics and engineering 195},
  25-28 (2006), 3224--3246.

\bibitem{hunter1976mechanics}
{\sc Hunter, S.~C.}
\newblock {\em Mechanics of continuous media}.
\newblock Halsted Press, 1976.

\bibitem{jang2020spatially}
{\sc Jang, Y.}
\newblock {\em {Spatially Continuous and Discontinuous Galerkin Finite Element
  Approximations for Dynamic Viscoelastic Problems}}.
\newblock PhD thesis, Brunel University London, 2020.
\newblock BURA http://bura.brunel.ac.uk/handle/2438/21084.

\bibitem{ozisik2010constants}
{\sc Ozisik, S., Riviere, B., and Warburton, T.}
\newblock {On the constants in inverse inequalities in $L_2$}.
\newblock Tech. rep., Rice University, 2010.

\bibitem{DG}
{\sc Rivi{\`e}re, B.}
\newblock {\em {Discontinuous Galerkin methods for solving elliptic and
  parabolic equations: theory and implementation}}.
\newblock SIAM, 2008.

\bibitem{riviere2003discontinuous}
{\sc Rivi\'ere, B., Shaw, S., Wheeler, M.~F., and Whiteman, J.~R.}
\newblock {Discontinuous Galerkin finite element methods for linear elasticity
  and quasistatic linear viscoelasticity}.
\newblock {\em Numerische Mathematik 95}, 2 (2003), 347--376.

\bibitem{DGV}
{\sc Rivi{\`e}re, B., Shaw, S., and Whiteman, J.}
\newblock {Discontinuous {G}alerkin finite element methods for dynamic linear
  solid viscoelasticity problems}.
\newblock {\em Numerical Methods for Partial Differential Equations 23}, 5
  (2007), 1149--1166.

\bibitem{VE}
{\sc Shaw, S., and Whiteman, J.}
\newblock {Some partial differential Volterra equation problems arising in
  viscoelasticity}.
\newblock In {\em Proceedings of Equadiff\/} (1998), vol.~9, pp.~183--200.

\bibitem{WARBURTON20032765}
{\sc Warburton, T., and Hesthaven, J.}
\newblock On the constants in hp-finite element trace inverse inequalities.
\newblock {\em Computer Methods in Applied Mechanics and Engineering 192}, 25
  (2003), 2765 -- 2773.

\bibitem{wihler2006locking}
{\sc Wihler, T.}
\newblock {Locking-free adaptive discontinuous Galerkin FEM for linear
  elasticity problems}.
\newblock {\em Mathematics of computation 75}, 255 (2006), 1087--1102.

\end{thebibliography}
\end{document}